\documentclass[11pt,twoside,reqno]{amsart}
\allowdisplaybreaks
\usepackage{amsmath,amstext,amssymb,epsfig}
\usepackage{graphicx}
\usepackage{mathrsfs}
\calclayout
\makeatletter
\renewcommand{\@seccntformat}[1]{\bf\csname the#1\endcsname.}
\renewcommand{\section}{\@startsection{section}{1}
	\z@{.7\linespacing\@plus\linespacing}{.5\linespacing}
	{\normalfont\upshape\bfseries\centering}}
\renewcommand{\@biblabel}[1]{\@ifnotempty{#1}{#1.}}
\makeatother
\theoremstyle{plain}
\newtheorem{thm}{Theorem}[section]
\newtheorem{lem}[thm]{Lemma}
\newtheorem{prop}[thm]{Proposition}
\newtheorem{cor}[thm]{Corollary}

\theoremstyle{definition}

\newtheorem{defn}[thm]{Definition}
\newtheorem{nt}[thm]{Note}
\newtheorem{rem}{Remark}[section]

\usepackage{cancel}
\usepackage[parfill]{parskip}
\usepackage[german]{varioref}
\usepackage[all]{xy}
\usepackage{color}

\def \>{\succ}
\def \<{\prec}

\begin{document}	
	\title[M. A. Fiidow\textsuperscript{1}, Ahmed Zahari \textsuperscript{2}, Bouzid Mosbahi \textsuperscript{3}]{  Quasi-Centroids and Quasi-Derivations of\\ Low Dimensional Associative Algebras}
	\author{ M. A. Fiidow\textsuperscript{1}, Ahmed Zahari \textsuperscript{2}, Bouzid Mosbahi \textsuperscript{3}}
\address{\textsuperscript{1}Department of Mathematics, Faculty of Science, Somali National University.}
		\address{\textsuperscript{2}Universit\'{e} de Haute Alsace, IRIMAS-D\'{e}partement de Math\'{e}matiques, 18, rue des Fr\`eres Lumi\`ere F-68093 Mulhouse, France.}
\address{\textsuperscript{3}University of Sfax, Faculty of Sciences of Sfax,  BP 1171, 3000 Sfax, Tunisia.}
\email{\textsuperscript{1}m.fiidow@snu.edu.so}
\email{\textsuperscript{2}zaharymaths@gmail.com}
\email{\textsuperscript{3}bouzidmosbahy@gmail.com}
	
	
	\date{\today}

\begin{abstract}  In this paper, we present some basic properties concerning the quasi-derivation algebra $QDer(\mathcal{A})$ and the quasi-centroid algebra $QC(\mathcal{A})$ of associative algebra $\mathcal{A}$. Furthermore,
using the result on classification of two, three and four dimensional associative algebra, we compute, for all two, three and four dimensional associative algebras, quasi-derivation algebra $QDer(\mathcal{A})$ and the quasi-centroid algebra $QC(\mathcal{A})$  algebras and give their corresponding dimension.
\end{abstract}

\maketitle
\section{Introduction}
The classification of associative algebra is a long-standing and recurring problem. A possible investigator \cite{peirc}oversaw the initial investigation. Since then, numerous fascinating discoveries on the issue have been made.\\
Further works  in this field can be found in  \cite{hazlett} (nilpotent algebras of dimension $\leq 4$ over $\mathbb{C}$), \cite{mazzola1979}- associative unitary algebras of dimension five over algebraic closed fields of characteristic not two, \cite{mazzola1980}- nilpotent commutative associative algebras of dimension $\leq 5$, over algebraic closed fields of characteristic not two, three and  recently, \cite{poonen2008}- nilpotent commutative associative algebras of dimension $\leq 5$, over algebraically closed fields and \cite{de} classify nilpotent associative algebras of dimensions $\leq 3$ over any field, and four-dimensional commutative nilpotent associative algebras over finite fields and over $\mathbb{R}$.\\

Derivation and generalised derivation algebras are essential topics in the study of Lie algebras and Lie superalgebras, as is well known. Centroids, quasi-centroids, generalised derivations, and derivation algebras of nilpotent Lie algebras all play essential roles in studying Levi factors \cite{Benoist}. Melville's work focused on the centroids of nilpotent Lie algebras \cite{Melville}.
Leger and Luks \cite{ZRY} conducted the most significant and methodical study on the generalised derivation algebras of a Lie algebra and associated Lie subalgebras. The quasi-derivation algebras and centroids, two friendly features of the generalised derivation algebras and their subalgebras. They looked into the generalised derivation algebras' structure and described Lie algebras that met specific criteria. They also noted, and the readers are directed to, that there are specific linkages between quasi-derivations and the cohomology of Lie algebras.\\
The current investigation focuses on the two, three and four-dimensional associative algebras quasi-derivation algebra $QDer(\mathcal{A})$ and quasi-centroid algebra $QC(\mathcal{A})$ using the classification of low dimensional  associative algebras listed in \cite{RRB}.
\section{Preliminaries}

\begin{defn}\label{dia}
An associative algebras is a $2$-tuple $(\mathcal{A}, \bullet)$ consisting of a  linear space $\mathcal{A}$  linear maps
 $\bullet : \mathcal{A}\times \mathcal{A} \longrightarrow \mathcal{A}$  satisfying, for all $x, y, z\in \mathcal{A}$ the following
 conditions :
\begin{eqnarray}
(x\bullet y)\bullet z&=&x\bullet(y\bullet z),\label{eq1}
\end{eqnarray}
\end{defn}
In an associative algebra we consider the right $R_{x},$ and the left $L_{x},$ multiplication operators defined as follows:
\begin{eqnarray}
R_{x}(y) = y\cdot x \\
 L_{x}(y) = x\cdot y
\end{eqnarray}
\begin{defn}
 A derivation of associative algebra $\mathcal{A}$ is a linear transformation \\$d:\mathcal{A} \rightarrow \mathcal{A}$ where
\begin{equation}d(x\cdot y)=d(x)y+xd(y) \quad\quad    \forall x,y\in \mathcal{A}\end{equation}
\end{defn}
We denote the set of all derivations of an associative algebra $\mathcal{A}$ by $Der(\mathcal{A}).$ The $Der(\mathcal{A})$ is an associative algebra with respect the composition operation $\circ$ and it is a Lie algebra with respect to the bracket $[d_1, d_2]=d_1\circ d_2 - d_2\circ d_1.$

In the foregoing, we give a few earlier results on some properties of the derivations
 of associative algebras. The proof of some of the facts given below can be found in \cite{FRS}.
\begin{prop}
Let $(\mathcal{A}, \cdot)$ be an associative algebra and $d\in Hom(\mathcal{A})$, then the following conditions are equivalent:\\
$$i. \ d\in Der(\mathcal{A})  \quad ii. \ [d, L_{x}]=L_{d(x)}  \quad \text{and} \quad  iii. \ [d, R_{x}]=R_{d(x)}$$
\end{prop}
\begin{thm}
If $d$ is a derivation of an associative algebra $\mathcal{A}$, then $d$ is a derivation of $\mathcal{A}$ as a lie algebra.
\end{thm}
\begin{lem} The sets $R(\mathcal{A})=\{R_{x} | x\in \mathcal{A} \}$, $L(\mathcal{A})=\{L_{x} | x\in \mathcal{A} \}$
 are subalgebras of the associative algebra $Der(A).$
\end{lem}
\begin{defn}
Let $\mathcal{A}$ be an arbitrary associative algebra over a field $\mathbb{K}.$ Then
\begin{center}
$\Lambda (\mathcal{A}) = \Big\{\phi\in End(\mathcal{A})|\phi(xy)=\phi(x)y= x\phi(y)\quad \forall x,y\in \mathcal{A}\Big\}$
\end{center}
is called centroid of $\mathcal{A}$
\end{defn}
\begin{nt}
Let $I$ be a nonempty subset of $\mathcal{A}.$ The subset  $$Z_{\mathcal{A}}(I)= \{x\in \mathcal{A}| xI= Ix =0\}$$ is said to be the
centralizer of $I$ in $\mathcal{A}$. In particular, if $I$ is an ideal of $\mathcal{A}$, then so is $Z_{\mathcal{A}}(I).$ In particular, if $I= \mathcal{A},$ write $Z_{\mathcal{A}}(I)= Z(\mathcal{A} ).$
\end{nt}

\begin{defn}
If $QC(\mathcal{A})=\Big\{\phi\in End(\mathcal{A})|\phi(x)\ast y=x\ast\phi(y) \Big\}$ for all $x$ and $y\in \mathcal{A}$, then $C(\mathcal{A})$ is called quasi-centroids of $\mathcal{A}$.
\end{defn}
\begin{defn}
If $ZDer(\mathcal{A})=\Big\{\phi\in End(\mathcal{A})|\phi(x)\ast y=x\ast\phi(y)=0 \Big\}$ for all $x$ and $y\in \mathcal{A}$, then $ZDer(\mathcal{A})$ is called central derivation of $\mathcal{A}$.
\end{defn}
\begin{defn}
If \, $\phi\in End(\mathcal{A})$ is said to be a quasi-derivation, if there exist $\phi^{'}\in End(\mathcal{A})$, such that:
$$\phi(x)y+x \phi(y)=\phi^{'}(xy)$$
\end{defn}
\begin{defn}
Let $\mathcal{A}$ be an associative algebra. We say that $\mathcal{A}$ is indecomposable if it can not be written as a direct sum of its ideals. Otherwise the $\mathcal{A}$ is called decomposable.
\end{defn}
\begin{defn}
Let $\mathcal{A}$ be an indecomposable associative  algebra. We say $\Lambda(\mathcal{A})$ is small if $\Lambda$ is generated by central derivations and the scalars.
\end{defn}
 The quasi-centroid of a decomposable associative algebra is  small if the quasi-centroids of each indecomposable factors are small.\\
 \section{Some properties of Quasi-Centroids and Quasi-Derivations of associative algebras}
 First we give some properties of center derivation algebras, quasi derivations algebras and quasi-centroid of associative algebras.
 \begin{lem} Let $\mathcal{A}$ be an associative algebra. Then\\
\begin{itemize}
  \item $Der(\mathcal{A})C(\mathcal{A})\subset C(\mathcal{A})$
  \item $[QDer(\mathcal{A}), QC(\mathcal{A})]\subset QC(\mathcal{A})$
  \item $[QC(\mathcal{A}), QC(\mathcal{A})]\subset QDer(\mathcal{A})$
  \item $C(\mathcal{A})\subset QDer(\mathcal{A})$
\end{itemize}
 \end{lem}
\begin{proof}
The conclusions can be easily obtained by the definitions of quasi-derivation and quasi-centroid.
\end{proof}
\begin{prop}
If $\mathcal{A}$ is an associative algebra, then $[C(\mathcal{A}),\ QC(\mathcal{A})]\subseteq End(\mathcal{A},\ Z(\mathcal{A}))$. Moreover, if $Z(\mathcal{A})=0$, then $[C(\mathcal{A}),\ QC(\mathcal{A})]= 0$.
\end{prop}
\begin{proof}
Assume that $\phi_1\in C(\mathcal{A}),\ \phi_2\in QC(\mathcal{A})$\textit{ }and for all $x,\ y\ \in \ \mathcal{A},$ we have:\textit{}
\[[\phi_1, \phi_2](x)\cdot y\ =\ \phi_1\ \phi_2(x)\cdot y-\phi_2\ \phi_1(x)\cdot y\]
\[=\ \phi_1(\phi_2(x)\cdot y)- \phi_1(x)\cdot\phi_2(y)\]
\[=\ \phi_1(\phi_2(x)\cdot y-x\cdot\phi_2(y))=0.\]
Hence, $[\phi_1,\ \phi_2](x)\in Z(\mathcal{A})$ and $[\phi_1,\ \phi_2]\in  End(A,\ Z(\mathcal{A}))$ as desired. Furthermore, if $Z(\mathcal{A})\ =\ \{0\}$, it is clear that$\ [C(\mathcal{A}),\ QC(\mathcal{A})]= \{0\}.$
\end{proof}

\begin{thm}
Let $\mathcal{A}$ be an associative algebra and $I$ be $\Lambda(\mathcal{A})$ invariant ideal of $\mathcal{A}$.
\begin{itemize}
\item[$1.$]$T(I)$ is a subspace of $W$ isomorphic to $V(I)$.
\item[$2.$] If $\Lambda(I)= \mathbb{K}id_{I}$ then $\Lambda(\mathcal{A})= \mathbb{K}id_{I}\oplus V(I)$ as a vector space.
\end{itemize}
\end{thm}
\begin{proof}
$(1)$. It is easy to  see  that $V(I)$ is an ideal of the associative algebra $\Lambda(\mathcal{A}).$ To prove $(1)$ consider the following map $\alpha: V(I)\rightarrow T(I),$ given by\\
 $$\alpha(\phi)(\bar{y})= \phi(y),$$ where $\phi\in V(I)$ and  $\bar{y}= y+I\in \mathcal{A}/I.$\\
The map $\alpha$ is well defined. Indeed, for if $\bar{y}= \bar{y_1}$, then $y-y_{1}\in I$ and so $\phi(y-y_{1})= 0.$ i.e $\phi(y)= \phi(y_1).$\\
It follows easily  that $\alpha$ is an injective. If $\alpha\phi(\bar{y})= \alpha\phi_{1}(\bar{y})$ for  $\phi, \phi_{1}\in \Lambda(\mathcal{A})$
and for any $y\in \mathcal{A}$, namely, $\phi(y)= \phi_{1}(y),$ then $\phi= \phi_{1}$. \\
We now show that $\alpha$ is onto. For every $f\in T(I)$, set $\phi_{f}: \mathcal{A}\rightarrow \mathcal{A}$, $\phi_{f}(x)= f(\bar{x})$, $\forall$ $x\in \mathcal{A}$ it follows from the definition of $\phi_{f}(xy)= f(\overline{xy})= f(\bar{x}\bar{y})=f(\bar{x})  \bar{y}= \bar {x}f(\bar {y})$ for all $x, y\in \mathcal{A}$ namely, $\phi_{f}(xy)= \phi_{f}(x)y= x\phi_{f}(y).$  Thus  $\phi_{f}\in \Lambda(\mathcal{A})$ and  so $\phi_{f}\in V(I)$ since $(\phi_{f})(I)=0$. But  $\alpha(\phi_{f})=f$ implies that $\alpha$ is onto. Its  fairly easy to see that $\alpha$ preserves operations on vector space from $\mathcal{A}/I$ to $Z_{\mathcal{A}}(I)$is linear. Thus $\alpha$ is an isomorphism of vector spaces.\\

Let us now prove part $(2)$. If $\Lambda(I)= \mathbb{K}id_{I}$, then for all   $\phi\in \Lambda(\mathcal{A}),$ $\psi|_{I}=  \lambda id_I,$ for some $\lambda\in \mathbb{K}.$ If $\phi\neq \lambda id_{\mathcal{A}}$, let $\psi(x)= \lambda x$, for all $x\in \mathcal{A},$ then $\psi\in\Lambda(\mathcal{A})$, $\phi- \psi\in V(I).$ Clearly, $\phi= \psi+ (\phi-\psi).$ Furthermore, $\mathbb{K}id_{\mathcal{A}}\cap V(I)=0$ so $\phi(\mathcal{A})= \mathbb{K}id_{\mathcal{A}}\oplus V(I).$
\end{proof}
\noindent

\section{Description of Quasi-Centroid of low-dimensional  associative algebras}\label{rr}
In this section describes the quasi-centroid  of  associative algebra with low dimensional over the field $\mathbb{F}.$
Let $\left\{e_1,e_2, e_3,\cdots, e_n\right\}$ be a basis of an $n$-dimensional  associative algebras $\mathcal{A}.$ The product of basis is expressed in terms of structure constants as follows
\begin{eqnarray}\label{qp}
\phi(e_i)=\sum_{j=1}^na_{ji}e_j\nonumber.
\end{eqnarray}

\begin{eqnarray}\label{eqQC}
\sum_{p=1}^na_{pi}\delta_{pj}^q=\sum_{p=1}^na_{pj}\delta_{ip}^q.
\end{eqnarray}
It is observed that if the structure constants $\{\gamma_{ij}^{k}\}$ of a associative algebra $\mathcal{A}$ are given then in order to describe its quasi-centroid one has to solve the system of equations above with respect to $a_{ij}$ $i,j=1,2,\cdots,n.$
\section{Description of Quasi-Centroids of two and three dimensional associative algebras}
This section is devoted to the description of the quasi-centroid of three-dimensional complex  associative algebras.  Here we make use of the  algorithm for finding the quasi-centroid, given in Section \ref{rr}. Here we  use  classification results of two and three-dimensional complex associative algebras from \cite{RRB}

\begin{thm}
The \textbf{Quasi-Centroid} of two dimensional complex  associative algebras are given as follows:
\end{thm}
\begin{center}
\small
\begin{tabular}{|c|c|c||c|c|c|c|c|} \hline
\textbf{$IC$} & \textbf{$QC(\mathcal{A})$} & \textbf{$Dim$} & \textbf{$IC$} & \textbf{$QC(\mathcal{A})$} & \textbf{$Dim$}\\
\hline
$\mathcal{A}s^{1}_{2}$&$\left(%
\begin{array}{cc}
  a_{11}&0\\
a_{21}&a_{22} \\
\end{array}%
\right)$&3&

$\mathcal{A}s^{2}_{2}$&$\left(%
\begin{array}{cc}
  a_{11} & 0 \\
  0 & a_{11} \\
\end{array}%
\right)$&1\\ \hline

$\mathcal{A}s^{3}_{2}$&$\left(%
\begin{array}{cc}
   a_{11} & 0 \\
  0 & a_{11} \\
\end{array}%
\right)$&1&

$\mathcal{A}s^{4}_{2}$&$\left(%
\begin{array}{cc}
  a_{11} & 0 \\
  0 & a_{22} \\
\end{array}%
\right)$&2\\
\hline

$\mathcal{A}s^{5}_{2}$&$\left(%
\begin{array}{cc}
   a_{11} & 0 \\
  a_{21}& a_{11} \\
\end{array}%
\right)$&2& & & \\
\hline
\end{tabular}
\end{center}
\begin{proof}
Consider $\mathcal{A}s_{2}^{1}$ from \cite{RRB}. The structure constants of $\mathcal{A}s_{2}^{1}$ are $\gamma_{11}^{2}=1$ and the others are zeros. Solving the system of equation (\ref{eqQC}), we have $a_{12}=0$ and $a_{22}=a_{11}.$\\
Therefore we obtain the centroids of $\mathcal{A}s_{2}^{1}$ in matrix form as follows $$QC(\mathcal{A})= \left\{\left(%
\begin{array}{cc}
  a_{11} & 0 \\
  a_{21} & a_{22} \\
\end{array}%
\right)| a_{11}, a_{21}\in\mathbb{C} \right\}.$$
\end{proof}

\begin{thm}
The \textbf{Quasi-Centroid} of three dimensional complex  associative algebras are given as follows:
\end{thm}
\begin{center}
\small
\begin{tabular}{|c|c|c||c|c|c|c|c|} \hline
\textbf{$IC$}&\textbf{$QC(\mathcal{A})$}&
\textbf{$Dim$}&\textbf{$IC$}&\textbf{$QC(\mathcal{A})$}&\textbf{$Dim$}\\
\hline
$\mathcal{A}s^{1}_{3}$&$\left(%
\begin{array}{ccc}
  a_{11}&0&a_{13}\\
a_{21}&a_{22}&a_{23}\\
a_{31}&0&a_{11}
\end{array}%
\right)$ & 6&
$\mathcal{A}s^{2}_{3}$&$\left(%
\begin{array}{ccc}
   a_{11}&0&0\\
a_{21}&a_{22}&a_{23}\\
0&0&a_{11} \\
\end{array}%
\right)$&$4$\\
\hline
$\mathcal{A}s^{3}_{3}$&$\left(%
\begin{array}{ccc}
   a_{11}&0&0\\
a_{21}&a_{11}&0\\
a_{31}&a_{32}&a_{33}\\
\end{array}%
\right)$&$5$&
$\mathcal{A}s^{4}_{3}$ &$\left(%
\begin{array}{ccc}
   a_{11}&a_{12}&a_{13}\\
a_{21}&g-a_{12}&-a_{13}\\
0&0&g
\end{array}%
\right)$&$4$\\
\hline
$\mathcal{A}s^{5}_{3}$ &$\left(%
\begin{array}{ccc}
  a_{11} & 0 & 0 \\
  0 & a_{11} & 0 \\
  0 & 0 & a_{11} \\
\end{array}%
\right)$&$1$&
$\mathcal{A}s^{6}_{3}$ &$\left(%
\begin{array}{ccc}
  a_{11}&a_{12}&a_{13}\\
a_{21}&g-a_{12}&-a_{13}\\
0&0&g
\end{array}%
\right)$&$4$\\
\hline
$\mathcal{A}s^{7}_{3}$&$\left(%
\begin{array}{ccc}
  a_{11}&0&0\\
0&a_{22}&0\\
0&0&a_{11}
\end{array}%
\right)$&$2$&
$\mathcal{A}s^{8}_{3}$ &$\left(%
\begin{array}{ccc}
  a_{11}&0&a_{13}\\
0&a_{11}&0\\
0&0&a_{11}
\end{array}%
\right)$&$2$\\
 \hline
 $\mathcal{A}s^{9}_{3}$ &$\left(%
\begin{array}{ccc}
  a_{11}&0&0\\
0&a_{11}&a_{23}\\
0&0&a_{11}\\
\end{array}%
\right)$&$2$&
$\mathcal{A}s^{10}_{3}$ &$\left(%
\begin{array}{ccc}
   a_{11}&0&a_{13}\\
0&a_{11}&a_{23}\\
0&0&a_{11}\\
\end{array}%
\right)$&$3$\\
\hline
$\mathcal{A}s^{11}_{3}$ &$\left(%
\begin{array}{ccc}
  a_{11}&a_{12}&a_{13}\\
a_{21}&g-a_{12}&a_{23}\\
0&0&a_{11}+a_{21}\\
\end{array}%
\right)$ & $5$&
$\mathcal{A}s^{12}_{3}$ & $\left(%
\begin{array}{ccc}
   a_{11}&0&a_{13}\\
a_{13}&a_{11}&a_{23}\\
0&0&a_{11}\\
\end{array}%
\right)$ & $3$\\
 \hline
 $\mathcal{A}s^{13}_{3}$ &$\left(%
\begin{array}{ccc}
  a_{11}&0&0\\
0&a_{22}&0\\
0&0&-a_{33}\\
\end{array}%
\right)$ &$3$&
$\mathcal{A}s^{14}_{3}$ &$\left(%
\begin{array}{ccc}
  a_{11}&a_{12}&0\\
0&a_{11}&0\\
0&0&a_{33} \\
\end{array}%
\right)$&$3$\\
\hline
$\mathcal{A}s^{15}_{3}$ &$\left(%
\begin{array}{ccc}
 a_{11}&0&0\\
0&a_{11}&0\\
0&0&a_{33} \\
\end{array}%
\right)$&$2$&
$\mathcal{A}^{16}_{3}$ &$\left(%
\begin{array}{ccc}
a_{11}&0&0\\
0&a_{11}&0\\
0&0&a_{33}\\
\end{array}%
\right)$&$2$\\
\hline
$\mathcal{A}^{17}_{3}$ &$\left(%
\begin{array}{ccc}
a_{11}&0&0\\
a_{21}&a_{22}&a_{23}\\
0&0&a_{33} \\
\end{array}%
\right)$&$5$& & & \\
\hline
\end{tabular}
\end{center}
where $g=a_{11}+ a_{21}$.\\
\begin{proof}
Consider Theorem \cite{RRB} which give the classifications of three-dimensional associative algebras. It is clear form  (\ref{eqQC})that the system of equations can applied to compute all the quasi-centroid of three-dimensional algebras.
The class $\mathcal{A}s_{3}^{1}$ has the structure constants as follows $\gamma_{13}^{2}=1, \gamma_{31}^{2}=1$ . By using condition (\ref{eqQC}), we have $a_{12}=a_{32}=0$ and $a_{33}=a_{11}.$
Therefore we obtain the quasi-centroid of $\mathcal{A}s_{3}^{1}$ in matrix form as follows
$$QC(\mathcal{A})= \left\{\left(%
\begin{array}{ccc}
  a_{11} & 0 & a_{13} \\
  a_{21} & a_{22} & a_{23} \\
  0 & 0 & a_{11} \\
\end{array}%
\right)| a_{11},  a_{21}, a_{23}\in\mathbb{C}\right\}.$$
\end{proof}

\section{Description of Quasi-Centroid  of four dimensional associative algebras}
In this section we make use the result\cite{RRB} on classification of four-dimensional associative algebras and present quasi-centroid of algebras.
\begin{thm}\label{theoQC2}
The \textbf{Quasi-Centroid}  of $4$-dimensional  associative algebras have the
following form
\end{thm}
  \centering
\begin{tabular}{||c||c||c||c||c||c||c||c||c||c||c||c||}
\hline
\textbf{$IC$}&\textbf{$QC(\mathcal{A})$} &\textbf{$Dim$}&\textbf{$IC$}&\textbf{$QC(\mathcal{A})$}&\textbf{$Dim$}\\
			\hline
$\mathcal{A}s_4^1$&
$\left(\begin{array}{cccc}
a_{11}&0&0&0\\
0&a_{22}&0&0\\
a_{31}&a_{32}&a_{33}&a_{43}\\
a_{41}&a_{42}&a_{43}&a_{44}\\
\end{array}
\right)$
&
10
&
$\mathcal{A}s_4^{2}$&
$\left(\begin{array}{cccc}
a_{11}&0&0&0\\
0&a_{11}&0&0\\
a_{31}&a_{32}&a_{33}&a_{43}\\
a_{41}&a_{42}&a_{43}&a_{44}\\
\end{array}
\right)$
&
9
\\ \hline
$\mathcal{A}s_4^3$&
$\left(\begin{array}{cccc}
a_{11}&0&0&0\\
a_{21}&a_{11}&0&0\\
a_{31}&0&a_{11}&0\\
a_{41}&a_{42}&a_{43}&a_{44}\\
\end{array}
\right)$
&
7
&
$\mathcal{A}s_4^{4}$&
$\left(\begin{array}{cccc}
a_{11}&0&0&0\\
0&a_{22}&0&0\\
0&0&a_{22}&0\\
a_{41}&a_{42}&a_{43}&a_{44}\\
\end{array}
\right)$
&
6
\\ \hline
$\mathcal{A}s_4^5$&
$\left(\begin{array}{cccc}
a_{11}&0&0&0\\
0&a_{22}&0&0\\
0&0&a_{22}&0\\
a_{41}&a_{42}&a_{43}&a_{44}\\
\end{array}
\right)$
&
6
&
$\mathcal{A}s_4^{6}$&
$\left(\begin{array}{cccc}
a_{11}&0&0&0\\
0&a_{11}&0&0\\
a_{31}&a_{32}&a_{33}&a_{34}\\
a_{41}&a_{42}&a_{43}&a_{44}\\
\end{array}
\right)$
&
9
\\ \hline
$\mathcal{A}s_4^7$&
$\left(\begin{array}{cccc}
a_{11}&0&0&0\\
a_{21}&a_{11}&0&0\\
a_{31}&a_{32}&a_{33}&a_{34}\\
a_{41}&a_{42}&a_{43}&a_{44}\\
\end{array}
\right)$
&
9
&
$\mathcal{A}s_4^{8}$&
$\left(\begin{array}{cccc}
a_{11}&0&0&0\\
0&a_{11}&0&0\\
0&0&a_{33}&0\\
a_{41}&a_{42}&a_{43}&a_{44}\\
\end{array}
\right)$
&
6
\\ \hline
$\mathcal{A}s_4^9$&
$\left(\begin{array}{cccc}
a_{11}&0&0&0\\
0&2a_{11}&0&0\\
0&0&a_{33}&0\\
a_{41}&a_{42}&a_{43}&a_{44}\\
\end{array}
\right)$
&
6
&
$\mathcal{A}s_4^{10}$&
$\left(\begin{array}{cccc}
a_{11}&0&0&0\\
0&a_{11}&0&0\\
0&0&a_{11}&0\\
0&0&0&a_{11}\\
\end{array}
\right)$
&
1
\\ \hline
$\mathcal{A}s_4^{11}$&
$\left(\begin{array}{cccc}
a_{11}&0&0&0\\
0&a_{11}&0&0\\
0&0&a_{11}&0\\
0&0&0&a_{11}\\
\end{array}
\right)$
&
1
&
$\mathcal{A}s_4^{12}$&
$\left(\begin{array}{cccc}
a_{11}&0&0&0\\
0&a_{11}&0&0\\
0&0&a_{11}&0\\
0&0&0&a_{11}\\
\end{array}
\right)$
&
1
\\ \hline
$\mathcal{A}s_4^{13}$&
$\left(\begin{array}{cccc}
a_{11}&0&0&0\\
0&a_{11}&0&0\\
0&0&a_{33}&0\\
0&0&0&a_{33}\\
\end{array}
\right)$
&
2
&
$\mathcal{A}s_4^{14}$&
$\left(\begin{array}{cccc}
a_{11}&0&0&0\\
0&a_{11}&0&0\\
a_{31}&a_{32}&a_{33}&a_{34}\\
0&0&0&a_{11}\\
\end{array}
\right)$
&
5
\\ \hline
$\mathcal{A}s_4^{15}$&
$\left(\begin{array}{cccc}
a_{11}&0&0&0\\
0&a_{11}&0&0\\
0&0&a_{11}&0\\
0&0&0&a_{11}\\
\end{array}
\right)$
&
1
&
$\mathcal{A}s_4^{16}$&
$\left(\begin{array}{cccc}
a_{11}&0&0&0\\
a_{21}&a_{22}&0&0\\
-a_{21}&q_1&a_{11}&q_2\\
0&0&0&a_{11}\\
\end{array}
\right)$
&
4
\\ \hline
$\mathcal{A}s_4^{17}$&
$\left(\begin{array}{cccc}
a_{11}&0&0&0\\
0&a_{11}&0&0\\
0&0&a_{11}&0\\
0&0&0&a_{11}\\
\end{array}
\right)$
&
1
&
$\mathcal{A}s_4^{18}$&
$\left(\begin{array}{cccc}
a_{11}&0&0&0\\
0&a_{11}&0&0\\
0&0&a_{33}&0\\
0&0&0&a_{33}\\
\end{array}
\right)$
&
2
\\ \hline
$\mathcal{A}s_4^{19}$&
$\left(\begin{array}{cccc}
a_{11}&0&0&0\\
0&a_{11}&0&0\\
0&0&a_{11}&0\\
0&0&0&a_{11}\\
\end{array}
\right)$
&
1
&
$\mathcal{A}s_4^{20}$&
$\left(\begin{array}{cccc}
a_{11}&0&0&0\\
0&a_{22}&0&0\\
0&0&a_{33}&0\\
0&0&0&a_{44}\\
\end{array}
\right)$
&
4
\\ \hline
$\mathcal{A}s_4^{21}$&
$\left(\begin{array}{cccc}
a_{11}&0&0&0\\
a_{21}&a_{22}&0&0\\
q_3&0&a_{11}&0\\
a_{41}&a_{42}&a_{43}&a_{44}\\
\end{array}
\right)$
&
9
&
$\mathcal{A}s_4^{22}$&
$\left(\begin{array}{cccc}
a_{11}&0&0&0\\
0&a_{11}&0&0\\
a_{31}&a_{32}&a_{33}&a_{34}\\
a_{41}&a_{42}&a_{43}&a_{44}\\
\end{array}
\right)$
&
9
\\ \hline
$\mathcal{A}s_4^{23}$&
$\left(\begin{array}{cccc}
a_{11}&0&0&0\\
0&a_{11}&0&0\\
a_{31}&0&a_{33}&a_{34}\\
a_{41}&a_{42}&a_{43}&a_{44}\\
\end{array}
\right)$
&
8
&
$\mathcal{A}s_4^{24}$&
$\left(\begin{array}{cccc}
a_{11}&0&0&0\\
0&a_{22}&0&0\\
a_{31}&a_{32}&a_{33}&a_{34}\\
0&-q_1&0&a_{44}\\
\end{array}
\right)$
&
7
\\ \hline
\end{tabular}

\newpage
  \centering
\begin{tabular}{||c||c||c||c||c||c||c||c||c||c||c||c||}
\hline
\textbf{$IC$}&\textbf{$QC(\mathcal{A})$} &\textbf{$Dim$}&\textbf{$IC$}&\textbf{$QC(\mathcal{A})$}&\textbf{$Dim$}\\
			\hline
$\mathcal{A}s_4^{29}$&
$\left(\begin{array}{cccc}
a_{11}&0&0&0\\
0&a_{11}&0&0\\
0&0&a_{11}&0\\
0&0&0&a_{11}\\
\end{array}
\right)$
&
1
&
$\mathcal{A}s_4^{30}$&
$\left(\begin{array}{cccc}
a_{11}&0&0&0\\
0&a_{11}&0&0\\
0&0&a_{11}&0\\
0&0&0&a_{11}\\
\end{array}
\right)$
&
1
\\ \hline
$\mathcal{A}s_4^{31}$&
$\left(\begin{array}{cccc}
a_{11}&0&0&0\\
0&a_{11}&0&0\\
0&0&a_{11}&0\\
0&0&0&a_{11}\\
\end{array}
\right)$
&
1
&
$\mathcal{A}s_4^{32}$&
$\left(\begin{array}{cccc}
a_{11}&0&0&0\\
0&a_{22}&0&0\\
0&0&a_{22}&0\\
0&0&0&a_{22}\\
\end{array}
\right)$
&
2
\\ \hline
$\mathcal{A}s_4^{33}$&
$\left(\begin{array}{cccc}
a_{11}&0&0&0\\
0&a_{22}&0&0\\
0&0&a_{22}&0\\
0&0&0&a_{22}\\
\end{array}
\right)$
&
2
&
$\mathcal{A}s_4^{34}$&
$\left(\begin{array}{cccc}
a_{11}&0&0&0\\
0&a_{22}&0&0\\
0&0&a_{33}&0\\
0&0&a_{43}&a_{33}\\
\end{array}
\right)$
&
4
\\ \hline
$\mathcal{A}s_4^{35}$&
$\left(\begin{array}{cccc}
a_{11}&0&0&0\\
0&a_{11}&0&0\\
0&0&a_{33}&0\\
a_{41}&a_{42}&a_{43}&a_{44}\\
\end{array}
\right)$
&
6
&
$\mathcal{A}s_4^{36}$&
$\left(\begin{array}{cccc}
a_{11}&0&0&0\\
0&a_{11}&0&0\\
a_{31}&a_{32}&a_{33}&a_{34}\\
a_{41}&0&0&a_{11}\\
\end{array}
\right)$
&
6
\\ \hline
$\mathcal{A}s_4^{37}$&
$\left(\begin{array}{cccc}
a_{11}&0&0&0\\
0&2a_{11}&0&0\\
0&0&a_{11}&0\\
a_{41}&a_{42}&a_{43}&a_{44}\\
\end{array}
\right)$
&
5
&
$\mathcal{A}s_4^{38}$&
$\left(\begin{array}{cccc}
a_{11}&0&0&0\\
0&2a_{11}&0&0\\
0&0&a_{11}&0\\
a_{41}&a_{42}&a_{43}&a_{44}\\
\end{array}
\right)$
&
5
\\ \hline
$\mathcal{A}s_4^{39}$&
$\left(\begin{array}{cccc}
a_{11}&0&0&0\\
0&a_{11}&0&0\\
0&0&a_{11}&0\\
0&0&0&a_{11}\\
\end{array}
\right)$
&
1
&
$\mathcal{A}s_4^{40}$&
$\left(\begin{array}{cccc}
a_{11}&0&0&0\\
0&a_{11}&0&0\\
a_{31}&0&a_{11}&0\\
0&0&0&a_{11}\\
\end{array}
\right)$
&
2
\\ \hline
$\mathcal{A}s_4^{41}$&
$\left(\begin{array}{cccc}
a_{11}&0&0&0\\
a_{21}&a_{11}&0&0\\
0&0&a_{33}&0\\
0&0&a_{43}&a_{33}\\
\end{array}
\right)$
&
4
&
$\mathcal{A}s_4^{42}$&
$\left(\begin{array}{cccc}
a_{11}&0&0&0\\
0&a_{11}&0&0\\
0&0&a_{11}&0\\
0&0&0&a_{11}\\
\end{array}
\right)$
&
1
\\ \hline
$\mathcal{A}s_4^{43}$&
$\left(\begin{array}{cccc}
a_{11}&0&0&0\\
0&a_{11}&0&0\\
a_{31}&0&a_{11}&0\\
0&0&a_{42}&a_{11}\\
\end{array}
\right)$
&
3
&
$\mathcal{A}s_4^{44}$&
$\left(\begin{array}{cccc}
a_{11}&0&0&0\\
0&a_{11}&0&0\\
a_{31}&0&a_{11}&0\\
0&0&0&a_{11}\\
\end{array}
\right)$
&
2
\\ \hline
$\mathcal{A}_4^{45}$&
$\left(\begin{array}{cccc}
a_{11}&0&0&0\\
0&a_{11}&0&0\\
0&0&a_{11}&0\\
0&0&0&a_{11}\\
\end{array}
\right)$
&
1
&
$\mathcal{A}s_4^{46}$&
$\left(\begin{array}{cccc}
a_{11}&0&0&0\\
0&a_{11}&0&0\\
0&0&a_{11}&0\\
0&0&0&a_{11}\\
\end{array}
\right)$
&
1
\\ \hline
$\mathcal{A}s_4^{47}$&
$\left(\begin{array}{cccc}
a_{11}&0&0&0\\
0&a_{11}&0&0\\
0&0&a_{11}&0\\
0&0&0&a_{11}\\
\end{array}
\right)$
&
1
&
$\mathcal{A}s_4^{48}$&
$\left(\begin{array}{cccc}
a_{11}&0&0&0\\
0&a_{11}&0&0\\
0&0&a_{11}&0\\
0&0&0&a_{11}\\
\end{array}
\right)$
&
1
\\ \hline
$\mathcal{A}s_4^{49}$&
$\left(\begin{array}{cccc}
a_{11}&0&0&0\\
0&a_{11}&a_{23}&0\\
0&0&a_{11}&0\\
0&0&0&a_{11}\\
\end{array}
\right)$
&
2
&
$\mathcal{A}s_4^{50}$&
$\left(\begin{array}{cccc}
a_{11}&0&0&0\\
a_{21}&a_{11}&a_{23}&0\\
0&0&a_{11}&0\\
a_{41}&a_{42}&0&a_{11}\\
\end{array}
\right)$
&
5
\\ \hline
$\mathcal{A}s_4^{51}$&
$\left(\begin{array}{cccc}
a_{11}&0&0&0\\
a_{21}&a_{11}&0&0\\
a_{31}&0&a_{11}&0\\
a_{41}&0&0&a_{11}\\
\end{array}
\right)$
&
4
&
$\mathcal{A}s_4^{52}$&
$\left(\begin{array}{cccc}
a_{11}&0&0&0\\
a_{21}&a_{11}&0&0\\
0&0&a_{11}&0\\
a_{41}&a_{42}&0&a_{11}\\
\end{array}
\right)$
&
4
\\ \hline
\end{tabular}

 \centering
\begin{tabular}{||c||c||c||c||c||c||c||c||c||c||c||c||}
\hline
\textbf{$IC$}&\textbf{$QC(\mathcal{A})$} &\textbf{$Dim$}&\textbf{$IC$}&\textbf{$QC(\mathcal{A})$}&\textbf{$Dim$}\\
			\hline
$\mathcal{A}s_4^{53}$&
$\left(\begin{array}{cccc}
a_{11}&0&0&0\\
0&a_{22}&0&0\\
0&a_{32}&a_{22}&0\\
0&a_{42}&a_{32}&a_{22}\\
\end{array}
\right)$
&
4
&
$\mathcal{A}s_4^{54}$&
$\left(\begin{array}{cccc}
a_{11}&0&0&0\\
0&a_{11}&0&0\\
0&0&a_{11}&0\\
0&0&0&a_{11}\\
\end{array}
\right)$
&
1
\\ \hline
$\mathcal{A}s_4^{55}$&
$\left(\begin{array}{cccc}
a_{11}&0&0&0\\
0&a_{11}&0&0\\
a_{31}&0&a_{11}&0\\
0&0&0&a_{11}\\
\end{array}
\right)$
&
2
&
$\mathcal{A}s_4^{56}$&
$\left(\begin{array}{cccc}
a_{11}&0&0&0\\
a_{21}&a_{11}&0&0\\
0&0&a_{11}&0\\
0&0&a_{21}&a_{11}\\
\end{array}
\right)$
&
2
\\ \hline
$\mathcal{A}s_4^{57}$&
$\left(\begin{array}{cccc}
a_{11}&0&0&0\\
a_{21}&a_{11}&0&0\\
a_{31}&a_{21}&a_{11}&0\\
a_{41}&a_{31}&a_{21}&a_{11}\\
\end{array}
\right)$
&
4
&
$\mathcal{A}s_4^{58}$&
$\left(\begin{array}{cccc}
a_{11}&0&0&0\\
0&a_{11}&0&0\\
0&0&a_{11}&0\\
0&0&0&a_{11}\\
\end{array}
\right)$
&
1
\\ \hline
\end{tabular}

\begin{proof}	

Consider that $\mathcal{A}s_4^{1}$. Applying the systems of equations (\ref{eqQC}), we get
$a_{12}=a_{13}=a_{14}=a_{21}=a_{23}=a_{24}=0$. Hence, the quasi-centroids of $\mathcal{A}s_4^{1}$ are indicated as follows\\
$$QC(\mathcal{A})= \left\{\left(%
\begin{array}{cccc}
  a_{11}&0&0&0\\
0&a_{22}&0&0\\
a_{31}&a_{32}&a_{33}&a_{43}\\
a_{41}&a_{42}&a_{43}&a_{44}\\
\end{array}%
\right)| a_{11},  a_{22}, a_{31}, a_{31},a_{33},a_{41}, a_{43}, a_{44}\in\mathbb{C}\right\}.$$
The quasi-centroid of the remaining parts of dimension four associative algebras can be handled similarly, as illustrated above.
\end{proof}
\begin{cor}\,
\begin{itemize}
\item The dimensions of the quasi-centroid of $4$-dimensional associative algebras range between $1$ and $10$.
	\end{itemize}
\end{cor}
\begin{cor}
\begin{itemize}
\item[\emph{i)}]  Besides the types $\mathcal{A}s^{10}_{4}$, $\mathcal{A}s^{11}_{4}$, $\mathcal{A}s^{12}_{4}$, $\mathcal{A}s^{15}_{4}$, $\mathcal{A}s^{17}_{4}$, $\mathcal{A}s^{19}_{4}$, $\mathcal{A}s^{29}_{4}$, $\mathcal{A}s^{30}_{4}$, $\mathcal{A}s^{31}_{4}$, $\mathcal{A}s^{39}_{4}$,$\mathcal{A}s^{40}_{4}$,$\mathcal{A}s^{42}_{4}$, $\mathcal{A}s^{43}_{4}$, $\mathcal{A}s^{44}_{4}$, $\mathcal{A}s^{45}_{4}$, $\mathcal{A}s^{46}_{4}$, $\mathcal{A}s^{47}_{4}$, $\mathcal{A}s^{48}_{4}$, $\mathcal{A}s^{49}_{4}$, $\mathcal{A}s^{54}_{4}$, $\mathcal{A}s^{55}_{4}$, $\mathcal{A}s^{56}_{4}$, $\mathcal{A}s^{57}_{4}$ and $\mathcal{A}s^{58}_{4}$  any four-dimensional complex  isomorphism classes of associative  algebra  have a small quasi-centroid.
\item[\emph{ii)}]  The quasi-centroid of a four-dimensional complex non isomorphic  of associative algebra are not small.
\end{itemize}
\end{cor}
\begin{itemize}
	\item $q_1=a_{11}-a_{22},\quad q_2=a_{11}-a_{23},\quad q_3=a_{23}-a_{24}.$
\end{itemize}
\section{Procedure to find Quasi-Derivation of low dimensional associative algebras}
This section is devoted to the description of quasi-derivation of two and three-dimensional complex associative algebras.
Let $\{e_1, e_2, e_3, \cdots , e_n \}$ be  a basis of an $n$-dimensional associative algebras,  an element $\phi$ of the quasi-derivation $QDer(\mathcal{A})$ being a linear transformation of the vector space $\mathcal{A}$ is represented in a matrix form $(a_{ij})_{i,j=1,2,\cdots,n},$ i.e. $d(e_i)=\sum\limits_{j=1}^{n}a_{ji}e_{j},$ $i=1,2,\cdots,n.$ According to the definition of the quasi-derivation the entries $a_{ij}$ $i,j=1,2,\cdots,n,$ of the matrix $(a_{ij})_{i,j=1,2,\cdots,n}$

must satisfy the following systems of equations:
\begin{equation}\label{qaaq}
  \sum\limits_{k=1}^{n}\Bigg(d_{ki}C_{kj}^{p}+d_{kj}C_{ik}^{p}-C_{ij}^{k}d'_{pk}\Bigg)=0,\quad i=j=p=1,2,\cdots,n.
\end{equation}
\begin{thm}
The  \textbf{Quasi-Derivation} of two dimensional complex  associative algebras are given as follows:
\end{thm}
\centering
\begin{tabular}{||c||c||c||c||c||c||c||c||c||c||c||c||}
\hline
&$D$&$D{'}$
\\ \hline
$\mathcal{A}s_2^1$
&
$\begin{array}{ll}
\left(\begin{array}{cccc}
d_{11}&0\\
d_{21}&d_{22}
\end{array}
\right)
\end{array}$
&
$\begin{array}{ll}
\left(\begin{array}{cccc}
a_{11}&0\\
a_{21}&2d_{11}
\end{array}
\right)
\end{array}$
\\ \hline
$\mathcal{A}s_2^{2}$
&
$\begin{array}{ll}
\left(\begin{array}{cccc}
\frac{1}{2} a_{11}&0\\
a_{21}&-\frac{1}{2} a_{11}+a_{22}
\end{array}
\right)
\end{array}$
&
$\begin{array}{ll}
\left(\begin{array}{cccc}
a_{11}&0\\
a_{21}&a_{22}
\end{array}
\right)
\end{array}$
\\ \hline
$\mathcal{A}s_2^{3}$
&
$\begin{array}{ll}
\left(\begin{array}{cccc}
\frac{1}{2} a_{11}&0\\
a_{21}&-\frac{1}{2} a_{11}+a_{22}
\end{array}
\right)
\end{array}$
&
$\begin{array}{ll}
\left(\begin{array}{cccc}
a_{11}&0\\
0&a_{22}
\end{array}
\right)
\end{array}$
\\ \hline
$\mathcal{A}s_2^{4}$
&
$\begin{array}{ll}
\left(\begin{array}{cccc}
\frac{1}{2}a_{11}&0\\
0& \frac{1}{2} a_{22}
\end{array}
\right)
\end{array}$
&
$\begin{array}{ll}
\left(\begin{array}{cccc}
a_{11}&0\\
0&a_{22}
\end{array}
\right)
\end{array}$
\\ \hline
$\mathcal{A}s_2^{5}$
&
$\begin{array}{ll}
\left(\begin{array}{cccc}
d_{11}&0\\
0&d_{11}
\end{array}
\right)
\end{array}$
&
$\begin{array}{ll}
\left(\begin{array}{cccc}
2d_{11}&0\\
0&2d_{11}
\end{array}
\right)
\end{array}$
\\ \hline
\end{tabular}
\begin{thm}
The Quasi-derivation of three dimensional complex  associative algebras are given as follows:
\end{thm}
  \centering
  \begin{tabular}{||c||c||c||c||c||c||c||c||c||c||c||c||}
\hline
&$D$&$D^{'}$\\
			\hline
$\mathcal{A}s_3^{1}$
&
$\begin{array}{ll}
\left(\begin{array}{cccc}
d_{11}&0&d_{23}\\
d_{21}&d_{22}&0\\
0&0&a_{22}-d_{11}
\end{array}
\right)
\end{array}$
&
$\begin{array}{ll}
\left(\begin{array}{cccc}
a_{11}&0&a_{13}\\
a_{21}&a_{22}&a_{23}\\
a_{31}&0&a_{33}
\end{array}
\right)
\end{array}$
\\ \hline
$\mathcal{A}s_3^{2}$
&
$\begin{array}{ll}
\left(\begin{array}{cccc}
d_{11}&0&0\\
d_{21}&d_{22}&d_{23}\\
0&0&a_{22}-d_{11}
\end{array}
\right)
\end{array}$
&
$\begin{array}{ll}
\left(\begin{array}{cccc}
a_{11}&0&a_{13}\\
a_{21}&a_{22}&a_{23}\\
a_{31}&0&a_{33}
\end{array}
\right)
\end{array}$
\\ \hline
$\mathcal{A}s_3^{3}$
&
$\begin{array}{ll}
\left(\begin{array}{cccc}
d_{11}&0&0\\
d_{21}&d_{22}&0\\
d_{31}&d_{32}&d_{33}
\end{array}
\right)
\end{array}$
&
$\begin{array}{ll}
\left(\begin{array}{cccc}
a_{11}&0&0\\
a_{21}&2d_{11}&0\\
a_{31}&2d_{21}&d_{11}+d_{22}
\end{array}
\right)
\end{array}$
\\ \hline
$\mathcal{A}s_3^{4}$
&
$\begin{array}{ll}
\left(\begin{array}{cccc}
d_{11}&d_{12}&d_{13}\\
d_{21}&d_{11}+d_{21}&a_{23}-d_{13}\\
0&0&a_{22}+d_{11}
\end{array}
\right)
\end{array}$
&
$\begin{array}{ll}
\left(\begin{array}{cccc}
a_{11}&0&0\\
a_{21}&a_{22}&a_{23}\\
a_{31}&0&2a_{22}-2d_{11}
\end{array}
\right)
\end{array}$
\\ \hline
$\mathcal{A}s_3^{5}$
&
$\begin{array}{ll}
\left(\begin{array}{cccc}
d_{11}&0&a_{13}\\
0&a_{22}+d_{11}&a_{23}\\
0&0&a_{11}-d_{11}
\end{array}
\right)
\end{array}$
&
$\begin{array}{ll}
\left(\begin{array}{cccc}
a_{11}&0&a_{13}\\
0&a_{22}&a_{23}\\
0&0&2a_{11}-2d_{11}
\end{array}
\right)
\end{array}$
\\ \hline
$\mathcal{A}s_3^{6}$
&
$\begin{array}{ll}
\left(\begin{array}{cccc}
d_{11}&d_{12}&d_{13}\\
d_{21}&d_{11}-d_{12}+d_{21}&a_{23}-d_{13}\\
0&0&a_{22}-d_{11}-d_{21}
\end{array}
\right)
\end{array}$
&
$\begin{array}{ll}
\left(\begin{array}{cccc}
a_{11}&0&0\\
a_{21}&a_{22}&a_{23}\\
a_{31}&0&2a_{22}-2d_{11}-2d_{21}
\end{array}
\right)
\end{array}$
\\ \hline
$\mathcal{A}s_3^{7}$
&
$\begin{array}{ll}
\left(\begin{array}{cccc}
d_{11}&d_{12}&d_{13}\\
0&a_{11}-d_{11}&0\\
0&0&a_{11}-d_{11}
\end{array}
\right)
\end{array}$
&
$\begin{array}{ll}
\left(\begin{array}{cccc}
a_{11}&a_{12}&-a_{12}\\
0&2a_{22}-2d_{11}&0\\
0&0&a_{11}-d_{11}
\end{array}
\right)
\end{array}$
\\ \hline
$\mathcal{A}s_3^{8}$
&
$\begin{array}{ll}
\left(\begin{array}{cccc}
d_{11}&0&\frac{1}{2} a_{13}\\
0&-a_{11}+a_{22}+d_{11}&a_{23}\\
0&0&a_{11}-d_{11}
\end{array}
\right)
\end{array}$
&
$\begin{array}{ll}
\left(\begin{array}{cccc}
a_{11}&0&a_{13}\\
0&a_{22}&a_{23}\\
0&0&2a_{11}-2d_{11}
\end{array}
\right)
\end{array}$
\\ \hline
$\mathcal{A}s_3^{9}$
&
$\begin{array}{ll}
\left(\begin{array}{cccc}
d_{11}&0&a_{13}\\
0&a_{22}-a_{11}+d_{11}&\frac{1}{2} a_{23}\\
0&0&a_{11}-d_{11}
\end{array}
\right)
\end{array}$
&
$\begin{array}{ll}
\left(\begin{array}{cccc}
a_{11}&0&a_{13}\\
0&a_{22}&a_{23}\\
0&0&2a_{11}-2d_{11}
\end{array}
\right)
\end{array}$
\\ \hline
$\mathcal{A}s_3^{10}$
&
$\begin{array}{ll}
\left(\begin{array}{cccc}
d_{11}&a_{12}&\frac{1}{2}a_{13}\\
a_{21}&a_{22}-a_{11}+d_{11}&\frac{1}{2} a_{23}\\
0&0&a_{11}-d_{11}
\end{array}
\right)
\end{array}$
&
$\begin{array}{ll}
\left(\begin{array}{cccc}
a_{11}&a_{12}&a_{13}\\
a_{21}&a_{22}&a_{23}\\
0&0&2a_{11}-2d_{11}
\end{array}
\right)
\end{array}$
\\ \hline
\end{tabular}

  \centering
  \begin{tabular}{||c||c||c||c||c||c||c||c||c||c||c||c||}
\hline
&$D$&$D{'}$\\
			\hline	
$\mathcal{A}s_3^{11}$
&
$\begin{array}{ll}
\left(\begin{array}{cccc}
d_{11}&d_{12}&d_{13}\\
d_{21}&d_{11}-d_{12}+d_{21}&\frac{1}{2} a_{23}-d_{13}\\
0&0&a_{22}-d_{11}-d_{21}
\end{array}
\right)
\end{array}$
&
$\begin{array}{ll}
\left(\begin{array}{cccc}
a_{11}&0&0\\
a_{21}&a_{22}&a_{23}\\
a_{31}&0&2a_{22}-2d_{11}-2d_{21}
\end{array}
\right)
\end{array}$
\\ \hline			
$\mathcal{A}s_3^{15}$
&
$\begin{array}{ll}
\left(\begin{array}{cccc}
d_{11}&a_{12}&0\\
0&a_{11}-d_{11}&0\\
0&0&\frac{1}{2}a_{33}
\end{array}
\right)
\end{array}$
&
$\begin{array}{ll}
\left(\begin{array}{cccc}
a_{11}&a_{12}&0\\
0&2a_{11}-2d_{11}&0\\
0&0&a_{33}
\end{array}
\right)
\end{array}$
\\ \hline
$\mathcal{A}s_3^{16}$
&
$\begin{array}{ll}
\left(\begin{array}{cccc}
d_{11}&a_{12}&0\\
0&a_{11}-d_{11}&0\\
0&0&\frac{1}{2}a_{33}
\end{array}
\right)
\end{array}$
&
$\begin{array}{ll}
\left(\begin{array}{cccc}
a_{11}&a_{12}&0\\
0&2a_{11}-2d_{11}&0\\
0&0&a_{33}
\end{array}
\right)
\end{array}$
\\ \hline
$\mathcal{A}s_3^{17}$
&
$\begin{array}{ll}
\left(\begin{array}{cccc}
d_{11}&0&0\\
d_{21}&d_{22}&d_{23}\\
0&0&\frac{1}{2}a_{33}
\end{array}
\right)
\end{array}$
&
$\begin{array}{ll}
\left(\begin{array}{cccc}
a_{11}&0&0\\
a_{21}&2d_{11}&0\\
a_{31}&0&a_{33}
\end{array}
\right)
\end{array}$
\\ \hline
\end{tabular}

\section{Description of Quasi-Derivation  of four dimensional associative algebras}
In this section we make use the result\cite{RRB} on classification of four-dimensional associative algebras and present quasi-derivation of algebras
\begin{thm}\label{Aut3}
The  \textbf{Quasi-Derivation}  of $4$-dimensional associative algebras have the following form:
\end{thm}
\begin{table}[ht!]
\centering
\begin{tabular}{||c||c||c||c||c||c||c||c||c||c||c||c||}
\hline
&$D$&$D^{'}$\\
\hline
$\mathcal{A}s_{4}^{1}$

&
$\begin{array}{ll}
\left(\begin{array}{cccc}
d_{11}&0&0&0\\
0&d_{22}&0&0\\
d_{31}&d_{32}&d_{33}&d_{34}\\
d_{41}&d_{42}&d_{43}&d_{44}
\end{array}
\right)
\end{array}$
&
$\begin{array}{ll}
\left(\begin{array}{cccc}
a_{11}&a_{12}&0&0\\
a_{21}&a_{22}&0&0\\
a_{31}&a_{32}&2d_{11}&0\\
a_{41}&a_{42}&0&2d_{22}
\end{array}
\right)
\end{array}$
\\ \hline
$\mathcal{A}s_4^{2}$
&
$\begin{array}{ll}
\left(\begin{array}{cccc}
d_{11}&0&0&0\\
0&d_{22}&0&0\\
d_{31}&d_{32}&d_{33}&d_{34}\\
d_{41}&d_{42}&d_{43}&d_{44}
\end{array}
\right)
\end{array}$
&
$\begin{array}{ll}
\left(\begin{array}{cccc}
a_{11}&a_{12}&0&0\\
a_{21}&a_{22}&0&0\\
a_{31}&a_{32}&d_{11}+d_{22}&0\\
a_{41}&a_{42}&0&d_{11}+d_{22}
\end{array}
\right)
\end{array}$
\\ \hline
$\mathcal{A}s_4^{3}$
&
$\begin{array}{ll}
\left(\begin{array}{cccc}
d_{11}&0&0&0\\
0&d_{22}&0&0\\
0&0&d_{33}&0\\
d_{41}&d_{42}&d_{43}&d_{44}
\end{array}
\right)
\end{array}$
&
$\begin{array}{ll}
\left(\begin{array}{cccc}
a_{11}&a_{12}&0&0\\
a_{21}&a_{22}&0&0\\
a_{31}&a_{32}&d_{11}+d_{22}&0\\
a_{41}&a_{42}&0&d_{11}+d_{33}
\end{array}
\right)
\end{array}$
\\ \hline
$\mathcal{A}s_4^{4}$
&
$\begin{array}{ll}
\left(\begin{array}{cccc}
d_{11}&0&0&0\\
0& \frac{1}{2} a_{22}&0&0\\
0&a_{32}&-\frac{1}{2}a_{22}+a_{33}&0\\
d_{41}&d_{42}&d_{43}&d_{44}
\end{array}
\right)
\end{array}$
&
$\begin{array}{ll}
\left(\begin{array}{cccc}
a_{11}&0&0&0\\
a_{21}&a_{22}&0&0\\
a_{31}&a_{32}&a_{33}&0\\
a_{41}&0&0&2d_{11}
\end{array}
\right)
\end{array}$
\\ \hline
$\mathcal{A}s_4^{5}$
&
$\begin{array}{ll}
\left(\begin{array}{cccc}
d_{11}&0&0&0\\
0&\frac{1}{2}a_{22}&0&0\\
0&a_{32}&-\frac{1}{2}a_{22}+a_{33}&0\\
d_{41}&d_{42}&d_{43}&d_{44}
\end{array}
\right)
\end{array}$
&
$\begin{array}{ll}
\left(\begin{array}{cccc}
a_{11}&0&0&0\\
a_{21}&a_{22}&0&0\\
a_{31}&a_{32}&a_{33}&0\\
a_{41}&0&0&2d_{11}
\end{array}
\right)
\end{array}$
\\ \hline
$\mathcal{A}s_4^6$
&
$\begin{array}{ll}
\left(\begin{array}{cccc}
d_{11}&0&0&0\\
0&d_{11}&0&0\\
d_{31}&d_{32}&d_{33}&d_{34}\\
d_{41}&d_{42}&d_{43}&d_{44}
\end{array}
\right)
\end{array}$
&
$\begin{array}{ll}
\left(\begin{array}{cccc}
a_{11}&a_{12}&0&0\\
a_{21}&a_{22}&0&0\\
a_{31}&a_{32}&2d_{11}&0\\
a_{41}&a_{42}&0&2d_{11}
\end{array}
\right)
\end{array}$
\\ \hline
\end{tabular}
\end{table}

\centering
\begin{tabular}{||c||c||c||c||c||c||c||c||c||c||c||c||}
\hline
&$D$&$D^{'}$\\
\hline
$\mathcal{A}s_4^{7}$
&
$\begin{array}{ll}
\left(\begin{array}{cccc}
d_{11}&d_{12}&0&0\\
0&d_{22}&0&0\\
d_{31}&d_{32}&d_{33}&d_{34}\\
d_{41}&d_{42}&d_{43}&d_{44}
\end{array}
\right)
\end{array}$
&
$\begin{array}{ll}
\left(\begin{array}{cccc}
a_{11}&a_{12}&0&0\\
a_{21}&a_{22}&0&0\\
a_{31}&a_{32}&d_{11}+d_{22}&0\\
a_{41}&a_{42}&0&2d_{11}
\end{array}
\right)
\end{array}$
\\ \hline
$\mathcal{A}s_4^{8}$
&
$\begin{array}{ll}
\left(\begin{array}{cccc}
d_{11}&d_{12}&0&0\\
d_{21}&d_{22}&0&0\\
0&0&\frac{1}{2}d_{11}+\frac{1}{2}d_{22}&0\\
d_{41}&d_{42}&d_{43}&d_{44}
\end{array}
\right)
\end{array}$
&
$\begin{array}{ll}
\left(\begin{array}{cccc}
a_{11}&a_{12}&a_{13}&0\\
a_{21}&a_{22}&a_{23}&0\\
a_{31}&a_{32}&a_{33}&0\\
a_{41}&a_{42}&a_{43}&d_{11}+d_{22}
\end{array}
\right)
\end{array}$
\\ \hline
$\mathcal{A}s_4^{9}$
&
$\begin{array}{ll}
\left(\begin{array}{cccc}
d_{11}&d_{12}&0&0\\
0&d_{22}&0&0\\
d_{31}&d_{32}&d_{33}&d_{34}\\
d_{41}&d_{42}&d_{43}&d_{44}
\end{array}
\right)
\end{array}$
&
$\begin{array}{ll}
\left(\begin{array}{cccc}
a_{11}&a_{12}&0&0\\
a_{21}&a_{22}&0&0\\
a_{31}&a_{32}&2d_{22}&0\\
a_{41}&a_{42}&\frac{-2d_{12}}{-1+\alpha}&d_{11}+d_{22}
\end{array}
\right)
\end{array}$
\\ \hline
$\mathcal{A}s_4^{10}$
&
$\begin{array}{ll}
\left(\begin{array}{cccc}
\frac{1}{2}a_{11}&0&0&0\\
a_{21}&-\frac{1}{2}a_{11}+a_{22}&0&a_{24}\\
a_{31}&a_{32}&-\frac{1}{2}a_{11}+a_{33}&a_{34}\\
a_{41}&a_{42}&a_{43}&-\frac{1}{2}a_{11}+a_{44}
\end{array}
\right)
\end{array}$
&
$\begin{array}{ll}
\left(\begin{array}{cccc}
a_{11}&0&0&0\\
a_{21}&a_{22}&a_{23}&a_{24}\\
a_{31}&a_{32}&a_{33}&a_{34}\\
a_{41}&a_{42}&a_{43}&a_{44}
\end{array}
\right)
\end{array}$
\\ \hline
$\mathcal{A}s_4^{11}$
&
$\begin{array}{ll}
\left(\begin{array}{cccc}
\frac{1}{2}a_{11}&0&0&0\\
a_{21}&-\frac{1}{2}a_{11}+a_{22}&a_{23}&0\\
a_{31}&a_{32}&-\frac{1}{2}a_{11}+a_{33}&0\\
a_{41}&0&0&-\frac{1}{2}a_{11}+a_{44}
\end{array}
\right)
\end{array}$
&
$\begin{array}{ll}
\left(\begin{array}{cccc}
a_{11}&0&0&0\\
a_{21}&a_{22}&a_{23}&0\\
a_{31}&a_{32}&a_{33}&0\\
a_{41}&0&0&a_{44}
\end{array}
\right)
\end{array}$
\\ \hline
$\mathcal{A}s_4^{12}$
&
$\begin{array}{ll}
\left(\begin{array}{cccc}
\frac{1}{2}a_{11}&0&0&0\\
a_{21}&-\frac{1}{2}a_{11}+a_{22}&0&0\\
0&0&d_{33}&0\\
a_{41}&a_{42}&-a_{21}&p_1
\end{array}
\right)
\end{array}$
&
$\begin{array}{ll}
\left(\begin{array}{cccc}
a_{11}&0&a_{13}&0\\
a_{21}&a_{22}&a_{23}&0\\
0&0&a_{33}&0\\
a_{41}&a_{42}&a_{43}&p_2
\end{array}
\right)
\end{array}$
\\ \hline
$\mathcal{A}s_4^{13}$
&
$\begin{array}{ll}
\left(\begin{array}{cccc}
\frac{1}{2}a_{11}&0&0&0\\
a_{21}&-\frac{1}{2}a_{11}+a_{22}&0&0\\
0&0&\frac{1}{2}a_{33}&0\\
0&0&a_{43}&-\frac{1}{2}a_{33}+a_{44}
\end{array}
\right)
\end{array}$
&
$\begin{array}{ll}
\left(\begin{array}{cccc}
a_{11}&0&0&0\\
a_{12}&a_{22}&0&0\\
0&0&a_{33}&0\\
0&0&a_{43}&a_{44}
\end{array}
\right)
\end{array}$
\\ \hline
$\mathcal{A}s_4^{14}$
&
$\begin{array}{ll}
\left(\begin{array}{cccc}
\frac{1}{2}a_{11}&0&0&0\\
a_{21}&-\frac{1}{2}a_{11}+a_{22}&0&0\\
d_{31}&d_{32}&d_{33}&d_{34}\\
a_{32}&0&0&p_3
\end{array}
\right)
\end{array}$
&
$\begin{array}{ll}
\left(\begin{array}{cccc}
a_{11}&0&0&0\\
a_{21}&a_{22}&0&0\\
0&a_{32}&a_{33}&a_{21}\\
a_{32}&0&0&p_4
\end{array}
\right)
\end{array}$
\\ \hline
$\mathcal{A}s_4^{15}$
&
$\begin{array}{ll}
\left(\begin{array}{cccc}
\frac{1}{2}a_{11}&0&0&0\\
a_{21}&-\frac{1}{2}a_{11}+a_{22}&a_{23}&0\\
a_{31}&a_{32}&-\frac{1}{2}a_{11}+a_{33}&0\\
a_{41}&0&0&-\frac{1}{2}a_{11}+a_{44}
\end{array}
\right)
\end{array}$
&
$\begin{array}{ll}
\left(\begin{array}{cccc}
a_{11}&0&0&0\\
a_{21}&a_{22}&a_{23}&0\\
a_{31}&a_{32}&a_{33}&0\\
a_{41}&0&0&a_{44}
\end{array}
\right)
\end{array}$
\\ \hline
$\mathcal{A}s_4^{16}$
&
$\begin{array}{ll}
\left(\begin{array}{cccc}
\frac{1}{2}a_{11}&0&0&0\\
a_{21}&-\frac{1}{2}a_{11}+a_{22}&a_{23}&a_{24} \\
a_{31}&a_{32}&-\frac{1}{2}a_{11}+a_{33}&a_{34}\\
a_{41}&a_{42}&a_{43}&-\frac{1}{2}a_{11}+a_{44}
\end{array}
\right)
\end{array}$
&
$\begin{array}{ll}
\left(\begin{array}{cccc}
a_{11}&0&0&0\\
a_{21}&a_{22}&a_{23}&a_{24}\\
a_{31}&a_{32}&a_{33}&a_{34}\\
a_{41}&a_{42}&a_{43} &a_{44}
\end{array}
\right)
\end{array}$
\\ \hline
$\mathcal{A}s_4^{17}$
&
$\begin{array}{ll}
\left(\begin{array}{cccc}
\frac{1}{2}a_{11}&0&0&0\\
a_{21}&-\frac{1}{2}a_{11}+a_{22}&0&0\\
0&0&d_{33}&0\\
a_{41}&a_{42}&p_2
\end{array}
\right)
\end{array}$
&
$\begin{array}{ll}
\left(\begin{array}{cccc}
a_{11}&0&a_{13}&0\\
a_{21}&a_{22}&a_{23}&0\\
0&0&a_{33}&0\\
a_{41}&a_{42}&a_{43}&p_5
\end{array}
\right)
\end{array}$
\\ \hline
$\mathcal{A}s_4^{18}$
&
$\begin{array}{ll}
\left(\begin{array}{cccc}
\frac{1}{2}a_{11}&0&0&0\\
a_{21}&-\frac{1}{2}a_{11}+a_{22}&0&0\\
0&0&d_{33}&0\\
0&0&0&a_{44}-d_{33}
\end{array}
\right)
\end{array}$
&
$\begin{array}{ll}
\left(\begin{array}{cccc}
a_{11}&0&a_{13}&0\\
a_{21}&a_{22}&a_{23}&0\\
0&0&a_{33}&0\\
0&0&a_{43}&a_{44}
\end{array}
\right)
\end{array}$
\\ \hline
\end{tabular}

\newpage
\centering
\begin{tabular}{||c||c||c||c||c||c||c||c||c||c||c||c||}
\hline
&$D$&$D^{'}$\\
			\hline
$\mathcal{A}s_4^{19}$
&
$\begin{array}{ll}
\left(\begin{array}{cccc}
\frac{1}{2}a_{11}&0&0&0\\
0&\frac{1}{2}a_{22}&0&0\\
a_{31}&0&-\frac{1}{2}a_{11}+a_{33}&0\\
0&a_{42}&0&-\frac{1}{2}a_{22}+a_{44}
\end{array}
\right)
\end{array}$
&
$\begin{array}{ll}
\left(\begin{array}{cccc}
a_{11}&0&0&0\\
0&a_{22}&0&0\\
a_{31}&0&a_{33}&0\\
0&a_{42}&0&a_{44}
\end{array}
\right)
\end{array}$
\\ \hline
$\mathcal{A}s_4^{20}$
&
$\begin{array}{ll}
\left(\begin{array}{cccc}
\frac{1}{2}a_{11}&0&0&0\\
0&\frac{1}{2}a_{22}&0&0\\
0&0&\frac{1}{2}a_{33}&0\\
0&0&0&\frac{1}{2}a_{44}
\end{array}
\right)
\end{array}$
&
$\begin{array}{ll}
\left(\begin{array}{cccc}
a_{11}&0&0&0\\
0&a_{22}&0&0\\
0&0&a_{33}&0\\
0&0&0&a_{44}
\end{array}
\right)
\end{array}$
\\ \hline
$\mathcal{A}s_4^{21}$
&
$\begin{array}{ll}
\left(\begin{array}{cccc}
d_{11}&0&0&0\\
d_{21}&d_{22}&0&0\\
d_{31}&d_{21}&-d_{11}+2d_{22}&0\\
d_{41}&d_{42}&d_{43}&d_{44}
\end{array}
\right)
\end{array}$
&
$\begin{array}{ll}
\left(\begin{array}{cccc}
a_{11}&a_{12}&0&0\\
a_{21}&a_{22}&0&0\\
a_{31}&a_{32}&2d_{11}&0\\
a_{41}&a_{42}&2d_{31}&2d_{22}
\end{array}
\right)
\end{array}$
\\ \hline
$\mathcal{A}s_4^{22}$
&
$\begin{array}{ll}
\left(\begin{array}{cccc}
d_{11}&0&0&0\\
0&d_{11}&0&0\\
d_{31}&d_{32}&d_{33}&d_{34}\\
d_{41}&d_{42}&d_{43}&d_{44}
\end{array}
\right)
\end{array}$
&
$\begin{array}{ll}
\left(\begin{array}{cccc}
a_{11}&a_{12}&0&0\\
a_{21}&a_{22}&0&0\\
a_{31}&a_{32}&2d_{11}&0\\
a_{41}&a_{42}&0&2d_{11}
\end{array}
\right)
\end{array}$
\\ \hline
$\mathcal{A}s_4^{23}$
&
$\begin{array}{ll}
\left(\begin{array}{cccc}
d_{11}&0&0&0\\
0&d_{11}&0&0\\
d_{31}&d_{32}&d_{33}&d_{34}\\
d_{41}&d_{42}&d_{43}&d_{44}
\end{array}
\right)
\end{array}$
&
$\begin{array}{ll}
\left(\begin{array}{cccc}
a_{11}&a_{12}&0&0\\
a_{21}&a_{22}&0&0\\
a_{31}&a_{32}&2d_{11}&0\\
a_{41}&a_{42}&0&2d_{11}
\end{array}
\right)
\end{array}$
\\ \hline
$\mathcal{A}s_4^{24}$
&
$\begin{array}{ll}
\left(\begin{array}{cccc}
d_{11}&0&0&0\\
d_{21}&d_{11}&0&0\\
0&0&d_{11}&0\\
d_{41}&d_{42}&d_{43}&d_{44}
\end{array}
\right)
\end{array}$
&
$\begin{array}{ll}
\left(\begin{array}{cccc}
a_{11}&a_{12}&a_{13}&0\\
a_{21}&a_{22}&a_{23}&0\\
a_{31}&a_{32}&a_{33}&0\\
a_{41}&a_{42}&a_{43}&2d_{11}
\end{array}
\right)
\end{array}$
\\ \hline
$\mathcal{A}s_4^{25}$
&
$\begin{array}{ll}
\left(\begin{array}{cccc}
d_{11}&0&0&0\\
d_{21}&d_{22}&0&0\\
d_{31}&d_{32}&d_{33}&d_{34}\\
-\frac{1}{2}a_{34}+\frac{1}{2}d_{21}&0&0&d_{22}
\end{array}
\right)
\end{array}$
&
$\begin{array}{ll}
\left(\begin{array}{cccc}
a_{11}&a_{12}&0&0\\
a_{21}&a_{22}&0&0\\
a_{31}&a_{32}&d_{11}+d_{22}&a_{34}\\
a_{41}&a_{42}&0&2d_{11}
\end{array}
\right)
\end{array}$
\\ \hline
$\mathcal{A}s_4^{26}$
&
$\begin{array}{ll}
\left(\begin{array}{cccc}
\frac{1}{2}a_{11}&0&0&0\\
\frac{1}{2}a_{21}&-\frac{1}{2}a_{11}+a_{22}&0&0\\
a_{31}&0&-\frac{1}{2}a_{11}+a_{33}&a_{34}\\
a_{41}&0&a_{43}&-\frac{1}{2}a_{11}+a_{44}
\end{array}
\right)
\end{array}$
&
$\begin{array}{ll}
\left(\begin{array}{cccc}
a_{11}&0&0&0\\
a_{21}&a_{22}&0&0\\
a_{31}&0&a_{33}&a_{34}\\
a_{41}&0&a_{43}&a_{44}
\end{array}
\right)
\end{array}$
\\ \hline
$\mathcal{A}s_4^{27}$
&
$\begin{array}{ll}
\left(\begin{array}{cccc}
\frac{1}{2}a_{11}&0&0&0\\
\frac{1}{2}a_{21}&-\frac{1}{2}a_{11}+a_{22}&a_{23}&a_{24}\\
\frac{1}{2}a_{31}&a_{32}&-\frac{1}{2}a_{11}+a_{33}&a_{34}\\
\frac{1}{2}a_{41}&a_{42}&a_{43}&-\frac{1}{2}a_{11}+a_{44}
\end{array}
\right)
\end{array}$
&
$\begin{array}{ll}
\left(\begin{array}{cccc}
a_{11}&0&0&0\\
a_{21}&a_{22}&a_{23}&a_{24}\\
a_{31}&a_{32}&a_{33}&a_{34}\\
a_{41}&a_{42}&a_{43}&a_{44}
\end{array}
\right)
\end{array}$
\\ \hline
$\mathcal{A}s_4^{28}$
&
$\begin{array}{ll}
\left(\begin{array}{cccc}
\frac{1}{2}a_{11}&0&0&0\\
\frac{1}{2}a_{21}&-\frac{1}{2}a_{11}+a_{22}&0&0\\
a_{31}&0&-\frac{1}{2}a_{11}+a_{33}&a_{34}\\
a_{41}&0&a_{43}&-\frac{1}{2}a_{11}+a_{44}
\end{array}
\right)
\end{array}$
&
$\begin{array}{ll}
\left(\begin{array}{cccc}
a_{11}&0&0&0\\
a_{21}&a_{22}&0&0\\
a_{31}&0&a_{33}&a_{34}\\
a_{41}&0&a_{43}&a_{44}
\end{array}
\right)
\end{array}$
\\ \hline
$\mathcal{A}s_4^{29}$
&
$\begin{array}{ll}
\left(\begin{array}{cccc}
\frac{1}{2}a_{11}&0&0&0\\
0&\frac{1}{2}a_{11}&0&0\\
0&0&-\frac{1}{2}a_{11}+a_{33}&0\\
a_{41}&0&0&-\frac{1}{2}a_{11}+a_{44}
\end{array}
\right)
\end{array}$
&
$\begin{array}{ll}
\left(\begin{array}{cccc}
a_{11}&0&0&0\\
0&a_{11}&0&0\\
0&0&a_{33}&0\\
a_{41}&0&0&a_{44}
\end{array}
\right)
\end{array}$
\\ \hline
$\mathcal{A}s_4^{30}$
&
$\begin{array}{ll}
\left(\begin{array}{cccc}
\frac{1}{2}a_{11}&0&0&0\\
0&\frac{1}{2}a_{11}&0&0\\
a_{31}&0&-\frac{1}{2}a_{11}+a_{33}&0\\
0&a_{42}&0&-\frac{1}{2}a_{11}+a_{44}
\end{array}
\right)
\end{array}$
&
$\begin{array}{ll}
\left(\begin{array}{cccc}
a_{11}&0&0&0\\
0&a_{11}&0&0\\
a_{31}&0&a_{33}&0\\
0&a_{42}&0&a_{44}
\end{array}
\right)
\end{array}$
\\ \hline
\end{tabular}

\begin{tabular}{||c||c||c||c||c||c||c||c||c||c||c||c||}
\hline
&$D$&$D^{'}$\\
\hline
$\mathcal{A}s_4^{31}$
&
$\begin{array}{ll}
\left(\begin{array}{cccc}
\frac{1}{2}a_{11}&0&0&0\\
0&\frac{1}{2}a_{11}&0&0\\
a_{31}&-a_{31}&-\frac{1}{2}a_{11}+a_{33}&0\\
a_{41}&0&0&-\frac{1}{2}a_{11}+a_{44}
\end{array}
\right)
\end{array}$
&
$\begin{array}{ll}
\left(\begin{array}{cccc}
a_{11}&0&0&0\\
0&a_{11}&0&0\\
a_{31}&-a_{31}&a_{33}&0\\
a_{41}&0&0&a_{44}
\end{array}
\right)
\end{array}$
\\ \hline
$\mathcal{A}s_4^{32}$
&
$\begin{array}{ll}
\left(\begin{array}{cccc}
0&0&0&0\\
\frac{1}{2}a_{11}&\frac{1}{2}a_{11}&0&0\\
-a_{31}&-a_{31}&-\frac{1}{2}a_{11}+a_{33}&0\\
0&a_{42}&0&-\frac{1}{2}a_{11}+a_{44}
\end{array}
\right)
\end{array}$
&
$\begin{array}{ll}
\left(\begin{array}{cccc}
a_{11}&0&0&0\\
0&a_{11}&0&0\\
a_{31}&-a_{31}&a_{33}&0\\
0&a_{42}&0&a_{44}
\end{array}
\right)
\end{array}$
\\ \hline
$\mathcal{A}s_4^{33}$
&
$\begin{array}{ll}
\left(\begin{array}{cccc}
\frac{1}{2}a_{11}&0&0&0\\
0&\frac{1}{2}a_{22}&0&0\\
0&a_{32}&-\frac{1}{2}a_{22}+a_{33}&a_{32}\\
0&0&0&\frac{1}{2}a_{22}
\end{array}
\right)
\end{array}$
&
$\begin{array}{ll}
\left(\begin{array}{cccc}
a_{11}&0&0&0\\
0&a_{22}&0&0\\
0&a_{32}&a_{33}&a_{32}\\
0&0&0&a_{22}
\end{array}
\right)
\end{array}$
\\ \hline
$\mathcal{A}s_4^{34}$
&
$\begin{array}{ll}
\left(\begin{array}{cccc}
\frac{1}{2}a_{11}&0&0&0\\
0&\frac{1}{2}a_{22}&0&0\\
0&0&\frac{1}{2}a_{33}&0\\
0&0&\frac{1}{2}a_{43}&-\frac{1}{2}a_{33}+a_{44}
\end{array}
\right)
\end{array}$
&
$\begin{array}{ll}
\left(\begin{array}{cccc}
a_{11}&0&0&0\\
0&a_{22}&0&0\\
0&0&a_{33}&0\\
0&0&a_{43}&a_{44}
\end{array}
\right)
\end{array}$
\\ \hline
$\mathcal{A}s_4^{35}$
&
$\begin{array}{ll}
\left(\begin{array}{cccc}
\frac{1}{2}a_{44}&d_{12}&0&0\\
-d_{12}&\frac{1}{2}d_{44}&0&0\\
0&0&\frac{1}{2}d_{44}&0\\
d_{41}&d_{42}&d_{43}&d_{44}
\end{array}
\right)
\end{array}$
&
$\begin{array}{ll}
\left(\begin{array}{cccc}
a_{11}&a_{12}&a_{13}&0\\
a_{21}&a_{22}&a_{23}&0\\
a_{31}&a_{32}&a_{33}&0\\
a_{41}&a_{42}&a_{43}&a_{44}
\end{array}
\right)
\end{array}$
\\ \hline
$\mathcal{A}s_4^{36}$
&
$\begin{array}{ll}
\left(\begin{array}{cccc}
d_{11}&0&0&0\\
d_{21}&d_{22}&0&0\\
d_{31}&d_{21}&-d_{11}+d_{22}&0\\
d_{21}-a_{34}&-d_{31}&p_6&-d_{11}+2d_{22}
\end{array}
\right)
\end{array}$
&
$\begin{array}{ll}
\left(\begin{array}{cccc}
a_{11}&a_{12}&0&0\\
a_{21}&a_{22}&0&0\\
a_{31}&a_{32}&2d_{22}&a_{34}\\
a_{41}&a_{42}&0&2d_{11}
\end{array}
\right)
\end{array}$
\\ \hline
$\mathcal{A}s_4^{37}$
&
$\begin{array}{ll}
\left(\begin{array}{cccc}
d_{11}&0&0&0\\
d_{21}&d_{11}&0&0\\
0&-d_{21}&d_{11}&0\\
d_{41}&d_{42}&d_{43}&d_{44}
\end{array}
\right)
\end{array}$
&
$\begin{array}{ll}
\left(\begin{array}{cccc}
a_{11}&a_{12}&a_{13}&0\\
a_{21}&a_{22}&a_{23}&0\\
a_{31}&a_{32}&a_{33}&0\\
a_{41}&a_{42}&a_{43}&2d_{11}
\end{array}
\right)
\end{array}$
\\ \hline
$\mathcal{A}s_4^{38}$
&
$\begin{array}{ll}
\left(\begin{array}{cccc}
\frac{1}{2}a_{11}&0&0&0\\
a_{21}&-\frac{1}{2}a_{11}+a_{22}&0&0\\
\frac{1}{2}a_{31}&0&-\frac{1}{2}a_{11}+a_{13}&0\\
a_{41}&-\frac{1}{2}a_{31}+a_{42}&a_{21}&p_7
\end{array}
\right)
\end{array}$
&
$\begin{array}{ll}
\left(\begin{array}{cccc}
a_{11}&0&0&0\\
a_{21}&a_{22}&0&0\\
a_{31}&0&a_{33}&0\\
a_{41}&a_{42}&a_{21}&p_8
\end{array}
\right)
\end{array}$
\\ \hline
$\mathcal{A}s_4^{39}$
&
$\begin{array}{ll}
\left(\begin{array}{cccc}
\frac{1}{2}a_{11}&0&0&0\\
\frac{1}{2}a_{21}&-\frac{1}{2}a_{11}+a_{22}&a_{23}&0\\
\frac{1}{2}a_{31}&a_{32}&-\frac{1}{2}a_{11}+a_{33}&0\\
a_{41}&0&0&-\frac{1}{2}a_{11}+a_{44}
\end{array}
\right)
\end{array}$
&
$\begin{array}{ll}
\left(\begin{array}{cccc}
a_{11}&0&0&0\\
a_{21}&a_{22}&a_{23}&0\\
a_{31}&a_{32}&a_{33}&0\\
a_{41}&0&0&a_{44}
\end{array}
\right)
\end{array}$
\\ \hline
$\mathcal{A}s_4^{40}$
&
$\begin{array}{ll}
\left(\begin{array}{cccc}
\frac{1}{2}a_{11}&0&0&0\\
\frac{1}{2}a_{21}&-\frac{1}{2}a_{11}+a_{22}&a_{23}&0\\
\frac{1}{2}a_{31}&a_{32}&-\frac{1}{2}a_{11}+a_{33}&0\\
a_{41}&0&0&-\frac{1}{2}a_{11}+a_{44}
\end{array}
\right)
\end{array}$
&
$\begin{array}{ll}
\left(\begin{array}{cccc}
a_{11}&0&0&0\\
a_{21}&a_{22}&a_{23}&0\\
a_{31}&a_{32}&a_{33}&0\\
a_{41}&0&0&a_{44}
\end{array}
\right)
\end{array}$
\\ \hline
$\mathcal{A}s_4^{41}$
&
$\begin{array}{ll}
\left(\begin{array}{cccc}
\frac{1}{2}a_{11}&0&0&0\\
\frac{1}{2}a_{21}&-\frac{1}{2}a_{11}+a_{22}&0&0\\
0&0&\frac{1}{2}a_{33}&0\\
0&0&\frac{1}{2}a_{43}&-\frac{1}{2}a_{33}+a_{44}
\end{array}
\right)
\end{array}$
&
$\begin{array}{ll}
\left(\begin{array}{cccc}
a_{11}&0&0&0\\
a_{21}&a_{22}&0&0\\
0&0&a_{33}&0\\
0&0&a_{43}&a_{44}
\end{array}
\right)
\end{array}$
\\ \hline
$\mathcal{A}s_4^{42}$
&
$\begin{array}{ll}
\left(\begin{array}{cccc}
\frac{1}{2}a_{11}&0&0&0\\
0&\frac{1}{2}a_{11}&0&0\\
\frac{1}{2}a_{31}&0&-\frac{1}{2}a_{11}+a_{33}&0\\
a_{41}&-a_{41}&0&-\frac{1}{2}a_{11}+a_{44}
\end{array}
\right)
\end{array}$
&
$\begin{array}{ll}
\left(\begin{array}{cccc}
a_{11}&0&0&0\\
0&a_{11}&0&0\\
a_{31}&0&a_{33}&0\\
a_{41}&-a_{41}&0&a_{44}
\end{array}
\right)
\end{array}$
\\ \hline
\end{tabular}

\newpage
\begin{tabular}{||c||c||c||c||c||c||c||c||c||c||c||c||}
\hline
&$D$&$D^{'}$\\
\hline
$\mathcal{A}s_4^{43}$
&
$\begin{array}{ll}
\left(\begin{array}{cccc}
\frac{1}{2}a_{11}&0&0&0\\
a_{21}&\frac{1}{2}a_{11}+a_{22}&0&0\\
\frac{1}{2}a_{31}&0&\frac{1}{2}a_{11}+a_{33}&0\\
a_{41}&-\frac{1}{2}a_{31}+a_{42}&a_{21}&p_9
\end{array}
\right)
\end{array}$
&
$\begin{array}{ll}
\left(\begin{array}{cccc}
a_{11}&0&0&0\\
a_{21}&a_{22}&0&0\\
a_{31}&0&a_{33}&0\\
a_{41}&a_{42}&a_{21}&p_{10}
\end{array}
\right)
\end{array}$
\\ \hline
$\mathcal{A}s_4^{44}$
&
$\begin{array}{ll}
\left(\begin{array}{cccc}
\frac{1}{2}a_{11}&0&0&0\\
0&\frac{1}{2}a_{11}&0&0\\
\frac{1}{2}a_{31}&0&-\frac{1}{2}a_{11}+a_{33}&0\\
a_{41}&-a_{41}&0&-\frac{1}{2}a_{11}+a_{44}
\end{array}
\right)
\end{array}$
&
$\begin{array}{ll}
\left(\begin{array}{cccc}
a_{11}&0&0&0\\
0&a_{22}&0&0\\
a_{31}&0&a_{33}&0\\
a_{41}&-d_{41}&0&a_{44}
\end{array}
\right)
\end{array}$
\\ \hline
$\mathcal{A}s_4^{45}$
&
$\begin{array}{ll}
\left(\begin{array}{cccc}
\frac{1}{2}a_{11}&0&0&0\\
0&\frac{1}{2}a_{11}&0&0\\
a_{31}&-a_{31}&-\frac{1}{2}a_{11}+a_{33}&0\\
a_{41}&-a_{41}&0&-\frac{1}{2}a_{11}+a_{44}
\end{array}
\right)
\end{array}$
&
$\begin{array}{ll}
\left(\begin{array}{cccc}
a_{11}&0&0&0\\
0&a_{11}&0&0\\
a_{31}&-a_{31}&a_{33}&0\\
a_{41}&-a_{41}&0&a_{44}
\end{array}
\right)
\end{array}$
\\ \hline
$\mathcal{A}s_4^{46}$
&
$\begin{array}{ll}
\left(\begin{array}{cccc}
\frac{1}{2}a_{11}&0&0&0\\
0&\frac{1}{2}a_{11}&0&0\\
a_{31}&-a_{31}&-\frac{1}{2}a_{11}+a_{33}&0\\
0&0&0&-\frac{1}{2}a_{11}+a_{44}
\end{array}
\right)
\end{array}$
&
$\begin{array}{ll}
\left(\begin{array}{cccc}
a_{11}&0&0&0\\
0&a_{11}&0&0\\
a_{31}&-a_{31}&a_{33}&0\\
0&0&0&a_{44}
\end{array}
\right)
\end{array}$
\\ \hline
$\mathcal{A}s_4^{47}$
&
$\begin{array}{ll}
\left(\begin{array}{cccc}
\frac{1}{2}a_{11}&0&0&0\\
0&\frac{1}{2}a_{22}&0&0\\
0&\frac{1}{2}a_{32}&-\frac{1}{2}a_{22}+a_{33}&a_{34}\\
0&\frac{1}{2}a_{42}&a_{43}&-\frac{1}{2}a_{22}+a_{44}
\end{array}
\right)
\end{array}$
&
$\begin{array}{ll}
\left(\begin{array}{cccc}
a_{11}&0&0&0\\
0&a_{22}&0&0\\
0&a_{32}&a_{33}&a_{34}\\
0&a_{42}&a_{43}&a_{44}
\end{array}
\right)
\end{array}$
\\ \hline
$\mathcal{A}s_4^{48}$
&
$\begin{array}{ll}
\left(\begin{array}{cccc}
d_{11}&0&0&0\\
d_{21}&d_{22}&0&0\\
d_{31}&d_{32}&2d_{22}-d_{11}&0\\
d_{41}&d_{42}&d_{43}&d_{44}
\end{array}
\right)
\end{array}$
&
$\begin{array}{ll}
\left(\begin{array}{cccc}
a_{11}&0&0&0\\
a_{21}&2d_{11}&0&0\\
a_{31}&2d_{21}&d_{11}+d_{22}&0\\
a_{41}&2d_{31}&d_{21}+d_{32}&d_{44}
\end{array}
\right)
\end{array}$
\\ \hline
$\mathcal{A}s_4^{49}$
&
$\begin{array}{ll}
\left(\begin{array}{cccc}
d_{11}&0&0&0\\
d_{21}&a_{22}-a_{11}+d_{11}&a_{13}&0\\
0&0&2a_{11}-2d_{11}&0\\
0&0&-a_{21}&a_{22}-d_{11}
\end{array}
\right)
\end{array}$
&
$\begin{array}{ll}
\left(\begin{array}{cccc}
a_{11}&0&a_{31}&0\\
a_{21}&a_{23}&a_{23}&-a_{13}\\
0&0&a_{11}-d_{11}&0\\
0&0&-a_{21}&a_{22}-d_{11}
\end{array}
\right)
\end{array}$
\\ \hline
$\mathcal{A}s_4^{50}$
&
$\begin{array}{ll}
\left(\begin{array}{cccc}
\frac{1}{2}a_{11}&0&0&0\\
\frac{1}{2}a_{21}&-\frac{1}{2}a_{11}+a_{22}&0&0\\
a_{31}&0&-\frac{1}{2}a_{11}+a_{33}&0\\
\frac{1}{2}a_{41}&-\frac{1}{2}a_{21}a_{42}&0&2a_{22}-a_{11}
\end{array}
\right)
\end{array}$
&
$\begin{array}{ll}
\left(\begin{array}{cccc}
a_{11}&0&0&0\\
a_{21}&a_{22}&0&0\\
a_{31}&0&a_{33}&0\\
a_{41}&a_{42}&0&2a_{22}-a_{11}
\end{array}
\right)
\end{array}$
\\ \hline
$\mathcal{A}s_4^{51}$
&
$\begin{array}{ll}
\left(\begin{array}{cccc}
\frac{1}{2}a_{11}&0&0&0\\
\frac{1}{2}a_{21}&-\frac{1}{2}a_{11}+a_{22}&a_{23}&a_{24}\\
\frac{1}{2}a_{31}&a_{32}&-\frac{1}{2}a_{11}+a_{33}&a_{34}\\
\frac{1}{2}a_{41}&a_{42}&a_{43}&-\frac{1}{2}a_{11}+a_{44}
\end{array}
\right)
\end{array}$
&
$\begin{array}{ll}
\left(\begin{array}{cccc}
a_{11}&0&0&0\\
a_{21}&a_{22}&a_{23}&a_{24}\\
a_{31}&a_{32}&a_{33}&a_{34}\\
a_{41}&a_{42}&a_{43}&a_{44}
\end{array}
\right)
\end{array}$
\\ \hline
$\mathcal{A}s_4^{52}$
&
$\begin{array}{ll}
\left(\begin{array}{cccc}
\frac{1}{2}a_{11}&0&0&0\\
\frac{1}{2}a_{21}&-\frac{1}{2}a_{11}+a_{22}&0&0\\
a_{31}&0&-\frac{1}{2}a_{11}+a_{33}&0\\
\frac{1}{2}a_{41}&-\frac{1}{2}a_{21}+a_{42}&0&-\frac{3}{2}a_{11}+2a_{22}
\end{array}
\right)
\end{array}$
&
$\begin{array}{ll}
\left(\begin{array}{cccc}
a_{11}&0&0&0\\
a_{21}&a_{22}&0&0\\
a_{31}&0&a_{33}&0\\
a_{41}&a_{42}&0&2a_{22}-a_{11}
\end{array}
\right)
\end{array}$
\\ \hline
$\mathcal{A}s_4^{53}$
&
$\begin{array}{ll}
\left(\begin{array}{cccc}
\frac{1}{2}a_{11}&0&0&0\\
0&\frac{1}{2}a_{22}&0&0\\
0&\frac{1}{2}a_{32}&-\frac{1}{2}a_{22}+a_{33}&0\\
0&\frac{1}{2}a_{42}&-\frac{1}{2}a_{32}+a_{43}&-\frac{3}{2}a_{22}+2a_{33}
\end{array}
\right)
\end{array}$
&
$\begin{array}{ll}
\left(\begin{array}{cccc}
a_{11}&0&0&0\\
0&a_{22}&0&0\\
0&a_{32}&a_{33}&0\\
0&a_{42}&a_{43}&\frac{1}{2}a_{11}
\end{array}
\right)
\end{array}$
\\ \hline
$\mathcal{A}s_4^{54}$
&
$\begin{array}{ll}
\left(\begin{array}{cccc}
\frac{1}{2}a_{11}&0&0&0\\
0&\frac{1}{2}a_{11}&0&0\\
0&0&\frac{1}{2}a_{11}&0\\
0&0&0&\frac{1}{2}a_{11}
\end{array}
\right)
\end{array}$
&
$\begin{array}{ll}
\left(\begin{array}{cccc}
a_{11}&0&0&0\\
0&a_{11}&0&0\\
0&0&a_{11}&0\\
0&0&0&\frac{1}{2}a_{11}
\end{array}
\right)
\end{array}$
\\ \hline
\end{tabular}

\begin{tabular}{||c||c||c||c||c||c||c||c||c||c||c||c||}
\hline
&$D$&$D^{'}$\\
\hline
$\mathcal{A}s_4^{55}$
&
$\begin{array}{ll}
\left(\begin{array}{cccc}
\frac{1}{2}a_{11}&0&0&0\\
\frac{1}{2}a_{21}&-\frac{1}{2}a_{11}+a_{22}&0&0\\
\frac{1}{2}a_{31}&a_{32}&-\frac{1}{2}a_{11}+a_{33}&0\\
\frac{1}{2}a_{41}&-\frac{1}{2}a_{21}+a_{42}&a_{43}&-\frac{3}{2}a_{11}+2a_{22}
\end{array}
\right)
\end{array}$
&
$\begin{array}{ll}
\left(\begin{array}{cccc}
a_{11}&0&0&0\\
a_{21}&a_{22}&0&0\\
a_{31}&a_{32}&a_{33}&0\\
a_{41}&a_{42}&a_{43}&-a_{11}+2a_{22}
\end{array}
\right)
\end{array}$
\\ \hline
$\mathcal{A}s_4^{56}$
&
$\begin{array}{ll}
\left(\begin{array}{cccc}
\frac{1}{2}a_{11}&0&0&0\\
0&-\frac{1}{2}a_{11}+a_{22}&a_{23}&0\\
0&a_{32}&-\frac{1}{2}a_{11}+a_{33}&0\\
\frac{1}{2}a_{41}&a_{42}&a_{43}&p_{11}
\end{array}
\right)
\end{array}$
&
$\begin{array}{ll}
\left(\begin{array}{cccc}
a_{11}&0&0&0\\
0&a_{22}&a_{23}&0\\
0&a_{32}&a_{33}&0\\
a_{41}&a_{42}&a_{43}&-a_{11}+\frac{1}{2}a_{41}
\end{array}
\right)
\end{array}$
\\ \hline
$\mathcal{A}s_4^{57}$
&
$\begin{array}{ll}
\left(\begin{array}{cccc}
\frac{1}{2}a_{11}&0&0&0\\
\frac{1}{2}a_{21}&-\frac{1}{2}a_{11}+a_{22}&0&0\\
\frac{1}{2}a_{31}&-\frac{1}{2}a_{21}+a_{32}&-\frac{3}{2}a_{11}+2a_{22}&0\\
\frac{1}{2}a_{41}&-\frac{1}{2}a_{31}+a_{42}&-\frac{3}{2}a_{21}+2a_{32}&-\frac{5}{2}a_{11}+3a_{22}
\end{array}
\right)
\end{array}$
&
$\begin{array}{ll}
\left(\begin{array}{cccc}
a_{11}&0&0&0\\
a_{21}&a_{22}&0&0\\
a_{31}&a_{32}&t_1&0\\
a_{41}&a_{44}&t_2&t_3
\end{array}
\right)
\end{array}$
\\ \hline
$\mathcal{A}s_4^{58}$
&
$\begin{array}{ll}
\left(\begin{array}{cccc}
\frac{1}{2}a_{11}&0&0&0\\
\frac{1}{2}a_{21}&k_1&0&0\\
\frac{1}{2}a_{31}&a_{32}&k_3&0\\
\frac{1}{2}a_{41}&k_2&k_4&k_5
\end{array}
\right)
\end{array}$
&
$\begin{array}{ll}
\left(\begin{array}{cccc}
a_{11}&0&0&0\\
a_{21}&a_{22}&0&0\\
a_{31}&a_{32}&a_{22}-2a_{32}&0\\
a_{41}&a_{42}&a_{43}&k_6
\end{array}
\right)
\end{array}$
\\ \hline
\end{tabular}

\begin{rem}
\noindent In the above-displayed tables, the following notations are used :
\begin{itemize}
\item $p_1=-a_{11}+a_{22}+d_{33},\quad p_2=-\frac{1}{2}a_{11}+a_{22}+d_{33},\quad p_3=\frac{1}{2}a_{11}-a_{22}+a_{33},\, p_4=a_{11}+a_{22}+a_{33}$
\item $p_5=-\frac{1}{2}a_{11}+a_{22}+d_{33},\, p_6=d_{11}+d_{21}-d_{22},\, p_7=-\frac{3}{2}a_{11}+a_{22}+a_{33},\, p_8=-a_{11}+a_{22}+a_{33}$
\item $ p_9=-\frac{3}{2}a_{11}+a_{22}+a_{33},\,p_{10}=a_{22}+a_{11}+a_{33},
\, p_{11}=-a_{11}+\frac{1}{2}a_{41}$
	\item $t_1=-a_{11}-2a_{22},\quad t_{2}=-a_{21}+2a_{32},\quad t_{3}=2a_{11}+a_{22}.$
	\item $k_1=-\frac{1}{2}a_{11}+a_{22},\quad
 k_{2}=\frac{1}{2}a_{21}-\frac{a_{31}}{2}+a_{42}, \quad k_{3}=-\frac{1}{2}a_{11}+a_{22}-2a_{32}$
 \item$ k_{4}=-\frac{3a_{21}}{2}+d_{43},\quad k_5=-\frac{3a_{11}}{2}+2a_{22}-2a_{32},\quad k_6=-a_{11}+2a_{22}-2a_{32}.$
 \item$IC$: Class isomorphism.
\end{itemize}
\end{rem}


\begin{thebibliography}{999}

\bibitem{Benoist} Benoist, Y. (1988). La partie semi-simple de l'algebre des dérivations d'une algebre de Lie nilpotente. CR Acad. Sci. Paris, 307, 901-904.

\bibitem{de} De Graaf, W. A. (2010). Classification of nilpotent associative algebras of small dimension. arXiv preprint arXiv:1009.5339.

 \bibitem{FRS}Fiidow M.A.,Rakhimov I.S., Said Hussain S.K., Derivations and Centroids of Associative algebras. {\it IEEE Proceedings of International Conference on Research and Education in Mathematics (ICREM7)}. (2015), 227--232.


\bibitem{FM} Fiidow, M. A., Mohammed,N.F., Rakhimov,I.S., and Husain,S.K.S., On inner derivations of finite dimensional associative algebras, {\it Far east journal of Mathematical Science}, (2017), 10(102), 2177-2188.


\bibitem{hazlett} Hazlett, O. C. (1916). On the classification and invariantive characterization of nilpotent algebras. {\it American Journal of Mathematics}, 38(2):109-138.

\bibitem{11}
Makhlouf, A. and Zahari, A., 2020. Structure and Classification of Hom-Associative Algebras, Acta et commentationes universitis Tartuensis de mathematica,  24 (1), 79-102.

\bibitem{mazzola1980} Mazzola, G. (1980). Generic finite schemes and hochschild cocycles. {\it Commentarii Mathematici Helvetici}, 55(1):267-293.

\bibitem{mazzola1979} Mazzola, G.(1979). The algebraic and geometric classification of associative algebras of dimension five. {\it Manuscripta Mathematica}, 27(1):81-101.

\bibitem{Melville} Melville, D. J. (1992). Centroids of nilpotent Lie algabras. Communications in algebra, 20(12), 3649-3682.

\bibitem{peirc} Peirce, B. (1881). Linear associative algebra. American Journal of Mathematics, 4(1):97-229.


\bibitem{poonen2008} Poonen, B. (2008). Isomorphism types of commutative algebras of finite rank. {\it Computational Arithmetic Geometry}, 463:111-120.


\bibitem{RRB} Rakhimov I. S., Riskhboev I. M., Basri, W., Complete list of low-dimensional complex associative algebras. \emph{arXiv:0910.0932v2}.

\bibitem{16}Su, Y., Xu, X. and Zhang, H., 2000. Derivation-simple algebras and the structures of Lie algebras of Witt type. Journal of Algebra, 233(2), pp.642-662.

\bibitem{12}
 Zahari, A. and Bakayoko, I., 2023. On BiHom-Associative dialgebras, Open J. Math. Sci. (7),  96-117.
\bibitem{ZRY} Zhang, R., and Zhang, Y. (2010). Generalized derivations of Lie superalgebras. Communications in Algebra, 38(10), 3737-3751.
 \end{thebibliography}
\end{document}